\documentclass[reqno, 11pt]{amsart}
\usepackage{ifthen}

\usepackage{setspace}
\usepackage[a4paper, left=3cm, right=3cm, top=4cm]{geometry}
\usepackage{amssymb, amsthm, epsf, epsfig}
\usepackage{amsmath}
\usepackage{verbatim}
\usepackage{caption}
\usepackage{subcaption}
\usepackage{dsfont}
\usepackage{wrapfig}
\usepackage[shortlabels]{enumitem}
\usepackage[usenames,dvipsnames]{xcolor}
\usepackage{esvect}

\newtheorem{theo}{Theorem}[section]
\newtheorem{prop}{Proposition}[section]
\newtheorem{lemma}{Lemma}[section]
\newtheorem{cor}{Corollary}[section]
\newtheorem{obs}{Observation}

\theoremstyle{definition}
\newtheorem{df}{Definition}
\newtheorem{ex}{Example}

\theoremstyle{remark}
\newtheorem{rem}{Remark}

\newboolean{Evas_version}
\setboolean{Evas_version}{false}
\ifthenelse{\boolean{Evas_version}}
{

}
{
\setlength\marginparwidth{2.4cm}

}

\renewcommand{\(}{\left(}
\renewcommand{\)}{\right)} 
\newcommand{\beq}{\begin{equation}} 
\newcommand{\eeq}{\end{equation}} 
\newcommand{\bal}{\begin{align}} 
\newcommand{\eal}{\end{align}} 
\newcommand{\bals}{\begin{align*}} 
\newcommand{\eals}{\end{align*}} 
  
\newcommand{\bth}{\begin{theo}} 
\newcommand{\bl}{\begin{lemma}} 
\newcommand{\el}{\end{lemma}} 
\newcommand{\bp}{\begin{prop}} 
\newcommand{\ep}{\end{prop}} 
\newcommand{\bdf}{\begin{df}} 
\newcommand{\edf}{\end{df}} 
\newcommand{\brem}{\begin{rem}} 
\newcommand{\erem}{\end{rem}} 
\newcommand{\bnrem}{\begin{nrem}} 
\newcommand{\enrem}{\end{nrem}} 
\newcommand{\bex}{\begin{ex}} 
\newcommand{\eex}{\end{ex}} 
\newcommand{\bcor}{\begin{cor}} 
\newcommand{\ecor}{\end{cor}} 
\newcommand{\bpf}{\begin{proof}} 
\newcommand{\epf}{\end{proof}}

\def\({\left(} 
\def\){\right)}

\newcommand{\R}{\mathbb{R}}

\newcommand{\E}{\mathbb{E}}
\newcommand{\PP}{\mathbb{P}}
\newcommand{\V}{{\mathbb{V}\mbox{ar}}} 

\newcommand{\pt}[1]{\textbf{#1}}

\newcommand{\map}[1]{\textbf{#1}}

\numberwithin{equation}{section}

\title{Asymptotic normality of pattern occurrences in random maps}

\author{Michael Drmota$^\ast$, Eva-Maria Hainzl$^*$, Nick Wormald$^+$} 

\thanks{$^\ast$ TU Wien, Institute of Discrete Mathematics and Geometry,
Wiedner Hauptstrasse 8-10, A-1040 Vienna, Austria. michael.drmota@tuwien.ac.at, eva-maria.hainzl@tuwien.ac.at. Research 
supported by the Austrian  Science Fund FWF, projects P 35016 and F 100203.}
\thanks{$^+$ School of Mathematics, Monash University,
VIC 3800,
Australia. nick.wormald@monash.edu. Research supported by the Australian Laureate Fellowships grant FL120100125.}

\begin{document}

\begin{abstract}
The purpose of this paper is to study the limiting distribution of special {\it additive functionals} on random planar maps, namely the number of occurrences of a given {\it pattern}. The main result is a central limit theorem for these pattern counts in the case of pattern with a simple boundary. The proof relies on a combination of analytic and combinatorial methods together with a moment method due to Gao and Wormald~\cite{GaoWormald}. It is an important issue to handle the overlap structure of two pattern which is the main difficulty in the proof. 
\end{abstract}
\maketitle 

\section{Introduction}\label{sec:intro}

It is a classic problem in the theory of random graphs to study the appearance of substructures, like triangles, and their distribution as the size of the graph tends to $\infty$. In contrast, the study of submap occurrences in random planar maps is still active and featuring several open problems. 
It started in 1985 when Richmond, Robinson and Wormald studied 3-connected subtriangulations of 3-connected triangulations \cite{RichmondRobinsonWormald} and continued in \cite{RichmondWormald} and \cite{BenderGaoRichmond} where general 0-1 laws were proven. About 10 years later Gao and Wormald showed that there is a sharp concentration of the number of occurrences of given submaps in random planar triangulations \cite{GaoWormald2} and they also published an article on asymptotic normality determined by high moments of general interest \cite{GaoWormald}. In this, they proved a central limit theorem for patterns which cannot self intersect in random triangulations and other families of maps as an example but the main theorem in this work has proved useful in other contexts as well (see e.g. \cite{BliemKous}, \cite{GaoWang} and most recently \cite{ThevJan}).  In \cite{DrmotaStufler}, Drmota and Stufler proved that the expectation for the number of pattern occurrences in random planar maps adjusted with a (regular critical) Boltzmann distribution is linear. They also stated it as an open problem (Problem 3.2) whether these pattern counts satisfy (generally) a central limit theorem. Later Drmota, Noy and Yu proved a central limit theorem for face counts for faces with a simple boundary, which they called pure polygons in \cite{DrmotaNoyYu}. Again the result in this article was a mere example of an application of a more powerful structural result which we will exploit in this work as well.

The main difficulty in extending the known cases to all patterns with simple boundary is that the usual generating function approach for counting maps with a given number of occurrences of a given pattern breaks down seriously when the patterns can overlap: a new variable would be needed for each type of cluster of overlapping patterns, and there are an unbounded number of such overlapping cases. We avoid this by using a technique involving factorial moments, which boils down to counting maps with a given number of marked patterns, in clusters involving no more than 2 patterns each. Thus, these   1- and 2-pattern clusters cannot overlap and the required enumeration can be carried out using the ``usual'' methods.

The next section provides an introduction to planar maps, their enumeration with generating functions by the classic decomposition scheme by Tutte \cite{tutte_1963} and relevant theorems concerning asymptotic normality in the limit. In Section \ref{sec:koalas} we prove a central limit theorem for a specific pattern which we call \emph{koala}. While it is just an example in light of our main result, it captures very well the details we will need to consider for general patterns. Finally, Section \ref{sec:main} is dedicated to our main result which is a central limit theorem for the number of occurrences of a pattern with simple boundary in a random planar rooted map. Its proofs are collected in Section \ref{sec:proofs} and the last section gives a short overview of possible extensions of our results in various directions. 

\section{Preliminaries and more}\label{sec:asymp_normal}

In this section we give a short introduction to planar maps and some techniques to achieve asymptotic normality. We relate these topics to each other and prove several preliminary lemmas.
\subsection{Maps and Patterns}
We start this subsection by giving exact definitions of maps and patterns.
\begin{df}[Rooted planar maps]
    A \emph{planar map} is a connected planar graph (with loops and multiple edges allowed) embedded onto the sphere. If one of its edges is oriented, we call it \emph{rooted}. The oriented edge is further called the \emph{root edge} and the vertex from which the root edge is pointing away the \emph{root vertex}.
    A planar map separates the surface into several connected regions called \emph{faces}. The face to the left of the root edge is called the \emph{root face} or the \emph{exterior face}.
    The \emph{valency} of a face is the number of edges incident to it, bridges being counted twice. A face of valency $m$ is called an \emph{$m$-gon}.
    
    We define the \emph{boundary} of a rooted map as the set of all edges and vertices incident with the root face.
\end{df}

We count rooted maps up to root-preserving isomorphism according to their number of edges. There are several ways to do so (see for example \cite{Schaeffer} for an introduction on map enumeration). We will quickly review the classic decomposition scheme by Tutte \cite{tutte_1963}. The idea is to delete the root edge. Then  the map either disconnects into two independent maps or not. The root edge of the resulting map is defined so as to preserve the root face and root vertex (see Figure \ref{fig:map_decom}). In case the map disconnects, both components are assigned root edges, where in one component the root edge preserves the root vertex and the root face and in the other component the root vertex is the vertex where the deleted root edge pointed to.

In the case that the map does not disconnect after deletion of the root edge there are several maps which  would reduce to the same map after deletion of the root edge since the deleted root edge could have pointed to any vertex on the boundary of the obtained map. Therefore, when we translate the decomposition in Figure \ref{fig:map_decom}, we keep track of the root face valency and obtain the following functional equations for the generating functions $m_i(z), i\geq 0$ which enumerate maps with root face valency $i$ and where $z$ marks the number of edges :
\[
    m_0(z) = 1, \quad m_i(z) = z\sum_{k=0}^{i-2}m_k(z)m_{i-k-2}(z) + z\sum_{k\geq i-1} m_k(z).
\]
Subsequently, we introduce a catalytic variable $u$ which marks the root face valency and a bivariate generating function $M(z,u) = \sum_{i\geq 0}m_i(z)u^i$ and reduce the equation system above to a single \emph{discrete differential equation (DDE)}
\begin{equation}\label{eq:maps}
     M(z,u) = 1 + zu^2M(z,u)^2 + zu \frac{uM(z,u)-M(z,1)}{u-1}.
\end{equation}

\begin{rem}\label{rem:root_edge}
    The root edge decomposition implicitly assigns a canonical root edge and root vertex to each connected submap \map{s} of the map \map{m}. We can determine it recursively by repeating the following steps:
    \begin{itemize}
        \item If the root edge of \map{m} is not contained in \map{s}, then delete it. In case \map{m} disconnects, then additionally delete the component which does not contain \map{s}. The new root edge is the root edge defined by the decomposition process above. Repeat until \map{s} contains the root edge.
        \item If the root edge of \map{m} is contained in \map{s}, then the root edge and the root vertex are the canonical root edge/vertex of \map{s}. 
    \end{itemize}
\end{rem}
\begin{figure}
    \centering
    \includegraphics[width=0.65\textwidth]{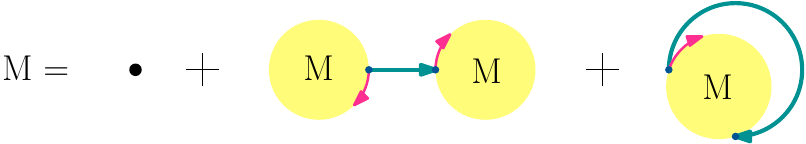}
    \caption{Planar map decomposition via root edge deletion}
    \label{fig:map_decom}
\end{figure}
This observation will be useful in the following sections.

When it comes to the definition of a patterns and submaps, there does not seem to be any consensus in the literature. Bender, Gao and Richmond~\cite{BenderGaoRichmond} define a submap occurrence essentially as follows.

\begin{df}[Submap occurrences]\label{def:submaps}
    Let $C$ be a cycle formed from some edges and their endpoints in a map $\map{m}$ on a surface $\map{S}$. (Recall that loops and multiple edges are permitted, so the cycle may have any length greater than 0.) Suppose that cutting along $C$ divides $\map{S}$ into two pieces. Duplicate $C$ so that each piece has a hole bounded by a copy of $C$ and then fill each of these holes with discs. This gives two surfaces $\map{S}_1$ and $\map{S}_1$ containing maps $\map{m}_1$ and $\map{m}_2$ with distinguished face formed by the added disc and therefore, distinguished boundary which are the edges of the cycle $C$.
    We say $\map{m}_1$ and $\map{m}_2$ are submaps in $\map{m}$ and the faces encircled by $C$ in $\map{m}$ are submap occurrences of $\map{m}_1$ and $\map{m}_2$ respectively.
\end{df}

 More recently, Yu~\cite{Yu} defined submaps as connected subsets of edges in a planar map and pattern occurrences in maps as follows.
\begin{df}[Pattern occurrences]\label{def:patterns}
    Let~\map{p} be a rooted map. We say that~\map{p} occurs as a pattern in a
    map~\map{m} if~\map{m} can be obtained by extending~\map{p} in the following way:
    \begin{enumerate}
        \item[(a)] adding vertices to the interior of the root face of~\map{p},
        \item[(b)] adding edges with their endpoints being either vertices from the
        boundary of~\map{p} or newly created vertices,
        \item[(c)] rerooting the map so obtained in such a way that its new root face is not contained in an interior face of~\map{p}.
    \end{enumerate}
\end{df}

In each case, an {\em occurrence} of a map $\map{p}$ in a map $\map{m}$ is equivalent to three incidence-preserving  injections from the vertices,  edges, and interior faces of $\map{p}$ to   the vertices,  edges, and interior faces of $\map{m}$ respectively. The occurrence is identified with the image of these bijections. However, any map, even a single vertex, can occur as a pattern in a map $\map{m}$, while only maps with a boundary consisting of a cycle may occur as a submap according to Definition~\ref{def:submaps}. We call this type of boundary a \emph{simple boundary}.
Pattern occurrences with simple boundary can be derived by rooting a submap so that its exterior face contains the root face of the larger map. In particular, we count pattern occurrences up to root-face preserving isomorphisms since the number of rooted pattern occurrences is then just a constant multiple of the number of (unrooted) pattern occurrences. 
In terms of submap occurrences, the number of occurrences of a pattern with simple boundary equals the number of submap occurrences which do not contain the root face.

\begin{df}[$S_i$-maps]
    The boundary of a map is called \emph{simple} if it contains as many edges as vertices.

    We will refer to maps with simple boundary and root face valency $i$ as $S_i$-maps and define their generating function as
    \[
        S_i(z) = \sum_{n\geq 0} s_{n,i}z^n
    \]
     where $s_{n,i}$ is the number of planar maps on $n$ edges with simple boundary and root face valency $i$.
\end{df}

We can decompose rooted maps in an alternative way using $S_i$-maps.  Each non-cut edge lies on a unique cycle of the boundary. Each vertex of such a cycle containing the root edge may be a cut vertex at which a (possibly empty) map is dangling off (see Figure \ref{fig:simple_dec}). In the case that the root edge is a cut edge, both of its vertices may be such cut vertices. The functional equation for an $S_i$-map can thus be obtained by rearranging the equation for a rooted map with root face valency~$i$.

Note that if we add a vector of variables $\pt{x}$ counting any set of non-root faces of valency $k$, possibly with restrictions on the shape of their boundary, or patterns with simple boundary, the decomposition will not be affected and we will not gain any further terms. Therefore, the following equation system holds for generating functions counting $S_i$-maps with additional variables counting either faces of valency $k$, possibly with restrictions on the shape of their boundary (see in particular Definition~8) or pattern occurrences with simple boundary .
\begin{figure}
    \centering
    \includegraphics[width=0.3\textwidth]{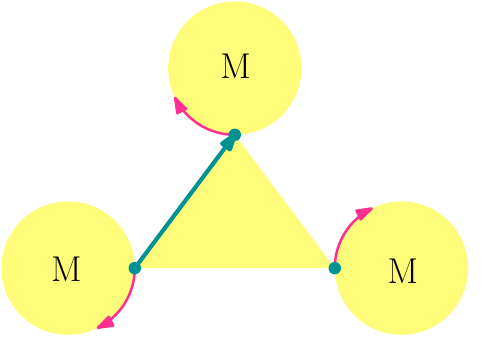}
    \caption{Decomposition of simple maps}
    \label{fig:simple_dec}
\end{figure}
\begin{lemma}\label{lem:simple}
    The generating functions $S_\ell(z,\pt{x})$ of maps with simple boundary and root face valency~$\ell$ with additional variables $\pt{x} = (x_1,x_2,\dots x_j)$ counting either patterns with simple boundary or faces of valency $k$ with restrictions to the shape of their boundary satisfy the following equation system
    \begin{align*}
        S_1(z,\pt{x}) &= M_1(z,\pt{x})\\
        S_2(z,\pt{x}) &= M_2(z,\pt{x})-M_1(z,u)^2-z\\
        S_\ell(z,\pt{x}) &= M_\ell(z,\pt{x}) - \sum_{k=1}^{\ell-1}S_k(z,\pt{x})[u^{\ell-k}]M(z,u,\pt{x})^k - z[u^{\ell-2}]M(z,u,\pt{x})^2, \quad \ell \geq 1
    \end{align*}
    where $M_i(z,\pt{x})$ denotes the generating function for maps with root face valency $i$ and $M(z,u,\pt{x})$ the bivariate generating function for rooted maps, where $u$ marks the root face valency and $z$ marks the number of edges.
\end{lemma}

\begin{obs}The equation system for $(S_\ell(z,\pt{x}))_{\ell \geq 1}$ is iterative and $[u^{\ell-k}]M(z,u,\pt{x})$ is a polynomial in $(M_i(z,\pt{x}))_{0\leq i \leq \ell}$. Therefore, we can express all $S_\ell(z,\pt{x})$ as polynomials in $(M_i(z,\pt{x}))_{0\leq i \leq \ell}$ as well.
\end{obs}

In the following, we are interested in pattern occurrences of maps with simple boundary or equivalently, submap occurrences according to Definition~\ref{def:submaps} which do not contain the root face. In turn, we will use the term submap rather freely describing a connected subset of faces in a map. The simplest map with simple boundary is of course just a single face with a simple boundary. We call such a face a \emph{simple polygon}. Counts of simple $k$-gons, as patterns contained in a random map, are proven to satisfy a central limit theorem in \cite{DrmotaNoyYu} by using the generating function of planar maps with a single additional variable counting the $k$-gons in the map. The functional equation for this generating function makes use of another class of maps which will be useful to us as well.

\begin{df}[$P_i$-maps]
    Let \map{m} be a map with root face valency $k > i$. If the first $i$ steps of the path along the boundary of \map{m} starting at its root vertex and in the direction determined by its root edge orientation consists of $i$ distinct edges and $i+1$ distinct vertices, we call the map a $P_i$-map and refer to it as a map with partial simple boundary of length $i$. We denote their generating function by
    $P_i(z,u)$ where $z$ marks the number of edges and $u$ the root face valency.
\end{df}
Clearly, an $S_\ell$-map is a $P_i$-map for all $1\leq i < \ell$. Maps with partial simple boundary were enumerated in \cite{Yu}.
\begin{figure}
    \centering
    \includegraphics[width=0.3\textwidth, page=2]{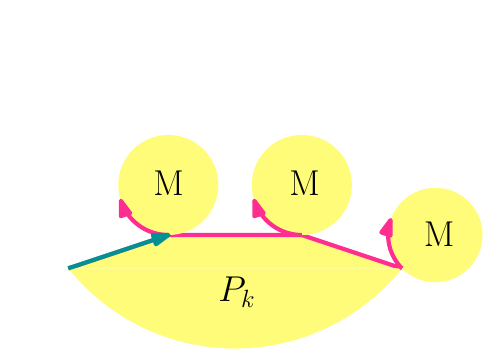}
    \caption{Decomposition of $P_i$-maps}
    \label{fig:dec_Pi}
\end{figure}
The idea to decompose a $P_i$-map is to consider the first $i$ steps of the contour walk along the boundary and make a case distinction. Either the path is simple (in this case we have a $P_i$-map) or the partial contour walk features cycles which are boundaries of maps. The walk can therefore be decomposed into a simple walk of length $k\leq i$ and contour walks along the boundary of maps that are attached at the vertices on the simple walk (see Figure~\ref{fig:dec_Pi}). Again, the decomposition is unaffected by variables counting faces of valency $k$, possibly with restrictions on the shape of their boundary or counting pattern occurrences with simple boundary and the following lemma obviously holds.
\begin{lemma}\label{lem:Yu}
    The generating functions $P_\ell(z,u,\pt{x})$ are given by
    \begin{align*}
        P_0(z,u,\pt{x}) &= M(z,u,\pt{x})\\
        P_{\ell}(z,u,\pt{x}) &= M(z,u,\pt{x})-\sum_{k=0}^{\ell-1}M_k(z,\pt{x})u^k - \sum_{k=0}^{\ell-1} P_{k}(z,u,\pt{x})u^{\ell-k}[u^{\ell-k}] M(z,u,\pt{x})^{k+1}
    \end{align*}
where $M_i(z,\pt{x})$ denotes the generating function for maps with root face valency $i$ and $M(z,u,\pt{x})$ the bivariate generating function for rooted maps, where $u$ marks the root face valency and $z$ marks the number of edges.
\end{lemma}

\subsection{Discrete differential equations, power laws and asymptotic normality}\label{sec:powerlaw}

We suppose that $\pt{X}_n = \left(X_n^{(1)}, \ldots, X_n^{(r)}\right)$, $n\ge 0$, is a sequence of random vectors
which distribution can be encoded in a generating function in several variables
\begin{equation}\label{eqyzu}
    F(z,{\bf x}) = \sum_{n\ge 0} f_n \E[ {\bf x}^{{\bf X}_n} ] z^n,
\end{equation}
where ${\bf x} = (x_1,\ldots,x_r)$ and $f_n > 0$ is a proper sequence. Such a situation appears for example, if $f_{n,m_1,\ldots,m_r}$ is a counting sequences that counts
objects of size $n$ such that (proper) parameters of this object have values $m_1,\ldots,m_r$. 
Let $F(z,{\bf x}) = F(z,x_1,\ldots,x_r)$ be the corresponding generating function
\[
    F(z,x_1,\ldots,x_r) = \sum_{n,m_1,\ldots,m_r} f_{n,m_1,\ldots,m_r} z^n x_1^{m_1}\cdots x_r^{m_r}.
\]
If we then define a sequence of random vectors $\pt{X}_n = \left(X_n^{(1)}, \ldots, X_n^{(r)}\right)$ by
\[
    \PP\left[\,\pt{X} = (m_1,\ldots, m_r)\right] = \frac{f_{n,m_1,\ldots,m_r}}{f_n},
\]
where 
\[
    f_n = \sum_{m_1,\ldots, m_r} f_{n,m_1,\ldots,m_r}
\]
is the $n$-th coefficient of the generating function $F(z,\pt{1}) = \sum_n f_n z^n$ counting the total number of objects of size $n$ then we obviously have (\ref{eqyzu}).
It is, thus, only necessary to use the generating function $F(z,{\bf x})$ in order characterize the
distribution of ${\bf X}_n$ since we have
\[
    \E[ {\bf x}^{{\bf X}_n} ] = \frac{[z^n]\, F(z,{\bf x}) }{[z^n]\, F(z,{\bf 1})},
\]
where $[z^n] A(z)$ denotes the $n$-th coefficient of a power series $A(z)$.

It turns out that under quite general assumptions the probability
generating function $\E[ {\bf x}^{{\bf X}_n} ]$ has the following 
asymptotic representation:
\begin{equation}\label{eq:power}
    \E[ {\bf x}^{{\bf X}_n} ] = e^{nf({\bf x}) + g({\bf x}) + O(1/n)},
\end{equation}
uniformly for $|x_j-1|\le \eta$ (for some $\eta > 0$) and 
analytic function $f$ and $g$ with $f({\bf 1})=g({\bf 1}) = 0$.

With the help of such an asymptotic relation (\ref{eq:power}) it can be shown 
that ${\bf X}_n$ satisfies a central limit theorem, see \cite{Hwang} or \cite[Theorem 2.22]{DrmotaTrees}.
Furthermore expected value and variance can  easily be determined:
\begin{align*}
    \E[{\bf X}_n] &= \mu  n + \left( g_{x_1}({\bf 1}),\ldots, g_{x_r}({\bf 1})  \right) +O(n^{-1}) \\
    \mathbb{C}{\rm ov}[{\bf X}_n] &= \Sigma  n + O(1).
\end{align*}
where $h_{y}$ denotes the partial derivative $\frac{\partial h}{\partial y}$ of the function $h(z, y)$ and
\[
        \mu := \left( f_{x_i}({\bf 1}) \right)_{1\le i\le r} \quad \mbox{and}\quad
        \Sigma := \left( f_{x_ix_j}({\bf 1}) + \delta_{i,j} f_{x_i}({\bf 1}) \right)_{1\le i,j\le r}.
    \] 
In the univariate case this reads as
\begin{align}\label{eq:exp_var_uni}
    \E[X_n]  &= f'(1) n + g'(1) + O(n^{-1}) \\
    \V[X_n]  &= (f'(1)+f''(1)) n +  O(1)\nonumber
\end{align}
and in particular, the second factorial moment equals
\begin{align}\label{eq:secmon_uni}
    \E[(X_n)_2]  &= f'(1)^2 n^2 +\left(f''(1)  +  2f'(1)g'(1)\right)n+ O(1).
\end{align}
In our context the objects will be (rooted) planar maps with $n$ edges and we will be interested in the distribution of the number of occurrences of certain faces; see Proposition~\ref{prop:nonself} below and Section~\ref{sec:main}. Their generating functions satisfy discrete differential equations (DDEs) of the form
\[
    M(z,u,\pt{x}) = zQ\big(z,u,M(z,u,\pt{x}),\Delta_{u=1} M(z,u,\pt{x})\big)+\sum_{i=1}^r (x_i-1)zq_i\big(z,u,M(z,u,\pt{x}),M(z,1,\pt{x})\big)
\]
where $\Delta_{u=1} M(z,u,\pt{x}) = \frac{M(z,u,\pt{x})-M(z,1,\pt{x})}{u-1}$, $Q$ is a polynomial and all $q_i$'s are analytic functions at the origin. In \cite[Theorem 4]{DrmotaNoyYu}, Drmota, Noy and Yu showed that the $(q_i)_{1\leq i\leq r}$ stay analytic up to $z=\rho(\pt{x})$ and $u=1$, where $\rho(\pt{x})$ is the radius of convergence of $M(z,1,\pt{x})$ depending on $\pt{x}$. We comment on this fact in more detail in Section \ref{sec:quasi_power}.

A straightforward application of this fact is a central limit theorem for patterns with simple boundary and the property that any interior face of a pattern occurrence cannot be the interior face of another pattern occurrence unless they have a root face preserving isomorphism. In this case we say the pattern cannot \emph{self-intersect}. 

\begin{ex}[Double glued triangles]\label{ex:dgtriang}
    We can apply Proposition \ref{prop:nonself}, which is stated below for example to the pattern of two triangles that are glued together along two edges. Clearly, this pattern cannot self-intersect since none of the edges on the boundary can be an interior edge of the pattern. See Figure \ref{fig:DGT} for an illustration of the pattern and the decomposition of the map as described in Proposition \ref{prop:nonself}. The equation then equals 
    \begin{align*}
        M(z,u,x) &= 1 + zu^2M(z,u,x)^2 + zu\frac{uM(z,u,x) -M(z,1,x)}{u-1} + (x-1)z^{3}P_1(z,u,x)\\
        &= 1 + zu^2M(z,u,x)^2 + zu\frac{uM(z,u,x) -M(z,1,x)}{u-1} \\
        &\qquad+ (x-1)z^{3}\big(M(z,u,x)-1 - zuM(z,1,x)M(z,u,x)\big)
    \end{align*}
    since $M_1(z,x) = zM(z,1,x)$ and $P_0(z,u,x)=M(z,u,x)$. As explained in the proof of Proposition \ref{prop:nonself} below the number of double glued triangles $X_n$ satisfies a central limit theorem and we can compute the expectation and variance by the formulas~(\ref{eq:exp_var_uni}) that say
    \begin{align*}
        \E[X_n]  = f'(1) n + O\left(1\right), \qquad \V[X_n]  = \left(f'(1)+f''(1)\right) n +  O(1),\nonumber
    \end{align*}
    where $f'(1) = -\frac{r'(1)}{r(1)}$, $f''(1) = \frac{r'(1)^2-r(1)r''(1)}{r(1)^2}$ and $r(x)$ is the radius of convergence of $M(z,1,x)$. The function $r(x)$ in turn can be computed as follows.

    Once again, we substitute $v=u-1$ and define $N(z,v,x) := M(z,v+1,x)-1$. Then the equation above translates to
    \begin{align*}
        N(z,v,x) &= z(v+1)^2(N(z,v,x)+1)^2+z(v+1)\left(N(z,v,x)+1+\Delta N(z,v,x)\right)\\
        &\qquad + (x-1)z^3\left(N(z,v,x)-z(v+1)(N(z,0,x)+1)(N(z,v,x)+1)\right).
    \end{align*}
    Next, we apply the method of Bousquet-Mélou and Jehanne~\cite{BMJ} and expand the DDE to the system
    \begin{align*}
        s_2 &= Q(r,s_1,x,s_2,s_3)\\
        s_1 &= s_1Q_g(r,s_1,x,s_2,s_3)+Q_f(r,s_1,x,s_2,s_3)\\
        s_3 &= Q_v(r,s_1,x,s_2,s_3)+s_3Q_g(r,s_1,x,s_2,s_3)\\
        0 & = \left|\left(
        \begin{matrix}
            Q_g-1 & Q_v & Q_f\\
            vQ_{gg} +Q_{fg}& Q_g+vQ_{gg}+Q_{fg}-1 & vQ_{fg}+Q_{ff}\\
            Q_{ug}+fQ_{gg} & Q_{vv}+fQ_{ug} & Q_{vf}+Q_{g}+fQ_{gf}-1
        \end{matrix}
        \right)\right|(r,s_1,x,s_2,s_3)
    \end{align*}
    where $|A| = \det (A)$ and
    \begin{align*} 
        Q(z,v,x,g,f) &= z(v+1)^2(g+1)^2+z(v+1)\left(g+1+f\right)\\
        &\qquad + (x-1)z^3\left(g-z(v+1)(g-vf+1)(g+1)\right).
    \end{align*}
    Its solutions $r(x),v(x),s_1(x),s_2(x),s_3(x)$ define the radius of convergence $r(x)$ of $M(z,1,x)$. Further, by the implicit function theorem, we can compute $r'(1)$ and $r''(1)$ and obtain
    \[
        r(1) = \frac{1}{12},\quad r'(1)= -\frac{7}{186624}, \quad r''(1) = \frac{11}{120932352}.
    \]
    Therefore we have
    \[
        f'(1) = \frac{7}{15552},\quad f'(1)+f''(1) =\frac{108649}{241864704}.
    \]
    The expectation of the number of double glued triangles was also computed in~\cite{Yu} by a different method and coincides with our value. 
\end{ex}
\begin{figure}
    \centering
    \begin{subfigure}[b]{0.45\textwidth}
        \centering
        \includegraphics[width=0.9\textwidth, page=1]{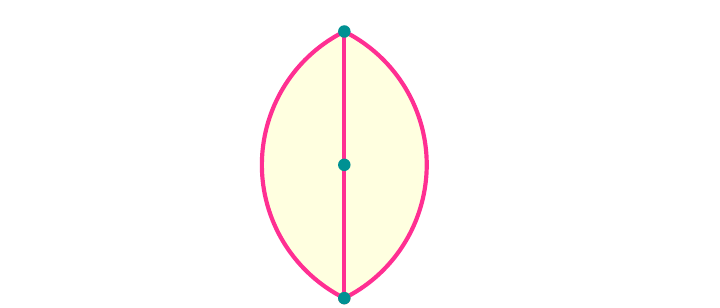}
        \vspace{4mm}
        \caption{The pattern 
        \emph{Double glued triangles}.}
        \vspace{4mm}
    \end{subfigure}
    \begin{subfigure}[b]{0.45\textwidth}
        \centering \includegraphics[width=0.9\textwidth, page=2]{dgtriangle.pdf}
        \caption{If the root edge is incident to a pattern occurrence, we count $P_1$-maps and attach two interior edges plus the new root edge.}
    
    \end{subfigure}
    \caption{Marking double glued triangles}
    \label{fig:DGT}
\end{figure}

For general non-intersecting patterns, the proof works analogously.
\begin{prop}\label{prop:nonself}
    Let \map{p} be a map with simple boundary which cannot self-intersect and $X_n$ be the (random) number of occurrences of \map{p} as a pattern up to rotations in a random rooted planar map on $n$ edges. Then
    \[
        \frac{X_n-\E(X_n)}{\sqrt{\V(X_n)}} \rightarrow \mathcal{N}(0,1)
    \]  
    where $\E(X_n) \sim \mu n$ and $\V(X_n) \sim \sigma^2 n$, for constants $\mu, \sigma^2 \in \R$.
\end{prop}
\begin{proof}
    First, we introduce an additional parameter $x$ which counts pattern occurrences of \map{p} and relate the recursive decomposition of planar maps to the generating function 
    \[
        M(z,u,x) = \sum_{j,k \geq 0} M_{j,k}(z)u^jx^k
    \]
    where $M_{j,k}(z)$ is the number of rooted planar maps with root face valency $j$ and $k$ occurrences of \map{p}.
    
    Recall that we count pattern occurrences up to isomorphisms preserving the root face of the pattern. We adapt the classic decomposition of maps by  deleting the root edge. If the root edge is incident to an occurrence of \map{p}, the deletion of the root edge would destroy exactly one pattern occurrence, as \map{p} cannot self-intersect. In this case, we delete not only the root edge but also the $i$ interior edges of the pattern occurrence and count these reduced maps. If $\ell$ is the root face valency of \map{p} and $i$ the number of interior edges of \map{p}, then we thus count maps with $n-i-1$ edges which have a partial simple boundary of length~$\ell-1$.
    
    When we attach the root edge and the interior of \map{p} again, let $r$ be the number of distinct ways to reinsert the interior of \map{p} into the face created by the root edge and the $P_{\ell-1}$-map (e.g. for double glued triangles $r=1$, but for triple glued pentagons $r=2$, see Figure~\ref{fig:DGT}). Clearly, $r$ is the number of patterns in the root-face preserving isomorphism class of~$\map{p}$. So the equations for the generating functions $M_{j,k}(z)$ which count maps with root face valency $j$ and $k$ pattern occurrences satisfy
    \begin{align*}
        M_{0,0}(z) &= 1, \quad M_{0,k} = 0, \quad k \geq 1\\
        M_{j,k}(z) &= z\sum_{s=0}^{j-2} \sum_{t=0}^k M_{s,t}(z)M_{j-s-2,t-k}(z) + z\sum_{s=j-1}^{\infty} M_{s,k}(z) \\
        &\qquad - z\sum_{s=j-1}^\infty r z^i P_{\ell-1, s,k}(z)+z\sum_{s=j-1}^\infty rz^iP_{\ell-1, s,k-1}(z)
    \end{align*}
    which we can sum up to a discrete differential equation by introducing the catalytic variable $u$ for the root face valency and the variable $x$ for the number of pattern occurrences. The system reduces therefore to the equation
    \begin{align*}
        M(z,u,x)&=1+zu^2M(z,u,x)^2+zu\frac{uM(z,u,x)-M(z,1,x)}{u-1} \\
        &\quad+ (x-1)z^{i+1}u^{1-\ell}P_{\ell-1}(z,u,x)
    \end{align*}
    where $P_{\ell}(z,u,x)$ is the generating function of planar maps with a partial simple boundary of length $\ell>1$. In~\cite{DrmotaNoyYu}  show that $P_{\ell}(z,u,x)$ is an analytic function in $z,x,M(z,1,x)$ at the radius of convergence $\rho(x)$ of $M(z,1,x)$ for $x$ close to $1$ and $u\leq 1$. As a consequence of~\cite[Theorem 3]{DrmotaNoyYu} the parameter $x$ satisfies a central limit theorem.
\end{proof}
Note that Proposition \ref{prop:nonself} includes in particular the asymptotic normality of appearances of \emph{simple polygons} in \cite{DrmotaNoyYu}. Other ``trivial" non-self-intersecting patterns are created by adding vertices along the interior edges of a fixed map \map{p} until no segment of the boundary can appear in any other occurrence of the pattern as one of the extended interior edges anymore. 

\begin{ex}[Triple glued pentagons]
    \emph{Triple glued pentagons} are maps with two pure $5$-gons which are adjacent along three edges (see Figure \ref{fig:TGT}). The equation which we obtain in the proof of Theorem~\ref{prop:nonself} equals 
    \begin{align*}
        M(z,u,x) &= 1 + zu^2M(z,u,x)^2 + zu\frac{uM(z,u,x) -M(z,1,x)}{u-1} + (x-1)2z^{4}P_3(z,u,x).
    \end{align*}
     Again, Proposition \ref{prop:nonself} gives immediately a central limit theorem for $X_n$ which denotes the number of triple glued pentagons in a uniform random map with $n$edges. By similar computations as in Example~\ref{ex:dgtriang} we obtain the constants 
    \[
        f'(1) = \frac{737}{34992000},\quad f''(1) =-\frac{15601553}{30611001600000000}
    \]
    and can further conclude that
    \[
        \E [X_n] = \frac{737}{34992000}n, \quad \mbox{and}\quad \V [X_n] = \frac{644711998447}{30611001600000000}n
    \]
\end{ex}

\begin{figure}
    \centering
    \begin{subfigure}[b]{0.45\textwidth}
        \centering
        \includegraphics[width=0.9\textwidth, page=3]{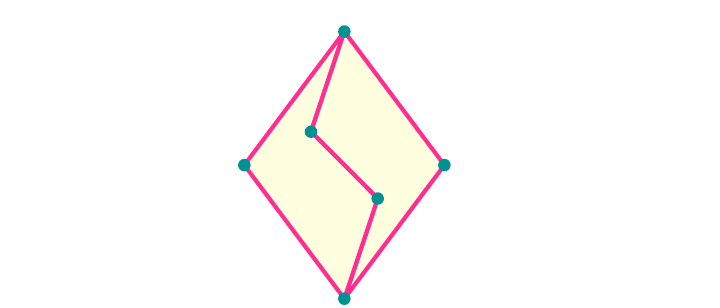}
        \vspace{4mm}
        \caption{The pattern 
        \emph{Triple glued pentagons}.}
        \vspace{4mm}
    \end{subfigure}
    \begin{subfigure}[b]{0.45\textwidth}
        \centering \includegraphics[width=0.9\textwidth, page=4]{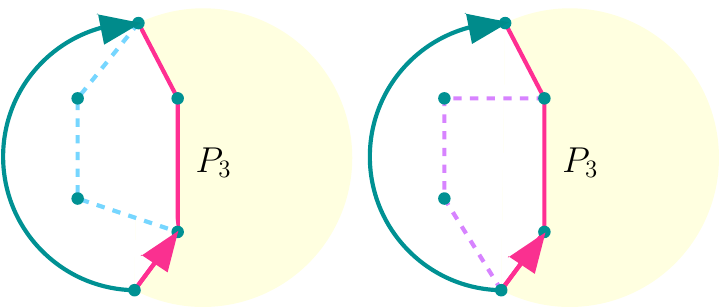}
        \caption{If the root edge is incident to a pattern, we count $P_3$-maps and attach the new root edge and three interior edges in two distinct ways.}
    
    \end{subfigure}
    \caption{Marking triple glued pentagons}
    \label{fig:TGT}
\end{figure}

\subsection{The moment method}\label{sec:momentmeth}
There is another interesting theorem on asymptotic normality by Gao and Wormald \cite{GaoWormald} that relies on the asymptotics of the factorial moments. As usual $(x)_k$ denotes the falling factorial $(x)_k = x(x-1)\cdots (x-k+1)$.

\begin{theo}[Gao, Wormald] \label{thm:gao_wormald}
    Suppose that $\mu_n$ and $\sigma_n$ are sequences of positive numbers with the following properties:
    \[
        \mu_n \to \infty, \quad \sigma_n \log^2 \sigma_n = o(\mu_n), \quad \mu_n = o(\sigma_n^3) \quad (n\to\infty).
    \]
    Suppose further that a sequence $(X_n)$ of non-negative integer valued random variables satisfy
    \begin{equation}\label{eqthGW}
        \E[(X_n)_k] \sim \mu_n^k \exp\left( \frac{k^2}2 \frac{ \sigma_n^2 - \mu_n}{\mu_n^2}  \right) 
    \end{equation}
    uniformly for all $k$ in the range $c\mu_n/\sigma_n  \le k \le c'\mu_n/\sigma_n$ for some constants 
    $c' > c > 0$ Then
    \[
        \frac{X_n - \mu_n}{\sigma_n} \to N(0,1).
    \]
\end{theo}

We show next that condition (\ref{eq:power}) is sufficient to obtain also (\ref{eqthGW}) and, thus,
we can check asymptotic normality also with this approach. So suppose that $\pt{X}_n = \left(X_n^{(1)},X_n^{(2)},\dots, X_n^{(m)}\right)$ is a sequence of random variables with the property 
    that the probability generating function $\E[ \pt{x}^{\pt{X}_n}]$ can be represented
    as
    \begin{equation*}\tag{\ref{eq:power}}
        \E[ \pt{x}^{\pt{X}_n}] = e^{nf(\pt{x}) + g(\pt{x}) + O(1/n)}
    \end{equation*}
    uniformly for complex $\pt{x}$ with $|x_i-1|<\eta$ (for some $\eta> 0$) and analytic functions
    $f(\pt{x})$ and $g(\pt{x})$ with $f(\pt{1}) = g(\pt{1}) = 0$ and $f_{x_i}(\pt{1})> 0$ for all $1\leq i \leq m$. 
    We recall that the expected value and variance are given by
    \[
        \E[\pt{X}_n] = f_\pt{x}(\pt{1})n + g_\pt{x}(\pt{1}) + O(1/n)\quad \mbox{and} \quad \mathbb{V}{\rm ar}[\pt{X}_n]= \Sigma n + O(1),
    \]
    where $h_{\pt{x}}$ denotes the gradient of the function $h$ and $(\Sigma)_{i,j} = f_{x_ix_j}(\pt{1}) + \delta_{i,j}f_{x_i}(\pt{1})$.

\begin{lemma}\label{LeMm}
    Suppose that the $m$-dimensional random vector $\pt{X}_n = \left(X_n^{(1)},X_n^{(2)},\dots, X_n^{(m)}\right)$ is given as above
and satisfies (\ref{eq:power}). Then we have
    \begin{equation}\label{eqLeMm}
        \E[(\pt{X}_n)_{\pt{k}}] =  \prod_{i=1}^m(nf_{x_i}(\pt{1}))^{k_i}
        \exp\left(\frac{1}{2n} \left<\pt{k},\Sigma  \pt{k} \right>\right) \left( 1 + O\left( \sum_{i=1}^m k_i^{3/2}/n \right) \right),
    \end{equation}
    where $\pt{k}=(k_1,k_2,\dots k_m)^T$, $(\Sigma)_{ij} = \frac{f_{x_ix_j}(\pt{1})}{f_{x_i}(\pt{1})f_{x_j}(\pt{1})}$ and the above holds uniformly for all $1\le k_i \le C \sqrt n$, where $C> 0$ is an arbitrary constant.
\end{lemma}
The proof of this lemma can be found in Section \ref{sec:proofs}.

Of course the lemma entails an alternative proof of Proposition \ref{prop:nonself} via Theorem \ref{thm:gao_wormald}. This approach certainly seems overly complicated in the case of patterns which cannot self-intersect but we will extend the approach in Section \ref{sec:main} to general patterns with simple boundary. In the general case it seems to be impossible to introduce a single counting variable for pattern occurrences like in Proposition \ref{prop:nonself}. The factorial moments on the other hand can be computed in several different manners. We will use the simple observation that
\begin{equation}\label{eq:kthmom}
    \E\left[(X_n)_k\right] = \sum_{i\geq k} (i)_k \frac{m_{n,i}}{m_n} = \frac{m_{n,k}^\circ}{m_n}
\end{equation}
where $m_n$ is the number of maps with $n$ edges, $m_{n,i}$ counts the number of maps with $n$ edges and $i$ pattern occurrences (in total) and $m^\circ_{n,k}$ is the number of maps with $n$ edges and $k$ pattern occurrences selected (among arbitrary many pattern occurrences) and labelled from $1$ to $k$. The labelled pattern occurrences   in general might intersect in all kinds of ways.
In the next step, we show that $m_{n,k}^\circ$ is asymptotically governed by the number of maps where labelled patterns intersect with at most one other labelled pattern. The intuition behind this lemma is that we will prove that $\sigma_n = \Theta(\sqrt{n})$ and there are almost surely a linear number of pattern occurrences~\cite{RichmondWormald}. Labelling only $k = \Theta(\mu_n/\sigma_n) = \Theta(\sqrt{n})$ of them, is thus such a sparse selection that it is rather unlikely that a labelled pattern occurrence would intersect with more than one other labelled pattern occurrence.

\begin{lemma}\label{lem:S1}
    Let \map{p} be a map with simple boundary and $\mu_n$ the expected number of occurrences of \map{p} in a random map with $n$ edges. Further, let $m^\circ_{n,k}$ be the number of all maps on $n$ edges with $k = \Theta(\sqrt{n})$ labelled occurrences of \map{p} and let $m^{\circ,\times}_{n,k}$ be the number of all such maps with at least $\mu_n/2$ pattern occurrences in total and where each labelled pattern occurrence intersects at most one other labelled occurrence of \map{p}. Then as $n$ tends to infinity, there is a constant $C>0$ such that for all $n$ large enough it holds that
    \[
        m_{m,k}^{\circ,\times} \leq m_{n,k}^\circ \leq  m_{m,k}^{\circ,\times} + \left(\frac{\mu_n}{2}\right)^k m_{n} + O\left(\frac{m_{n,k}^\circ}{\sqrt{n}}\right) 
    \]
\end{lemma}

The proof is based on elementary counting and is deferred to Section \ref{sec:proofs}. 

The great advantage of restricting our considerations to $m_{n,k}^{\circ, \times}$ is that we can introduce additional counting variables which mark single patterns and intersecting pairs. Let $i$ be the number of distinct ways that two pattern occurrences can intersect. Naturally $i$ is finite and so we can categorize the submaps of two intersecting pattern occurrences into  \emph{intersection types $1,2,\dots, i$}. Then we could set up a functional equation of the form
\begin{align*}
    M(z,u,\pt{x}) = 1 &+ zu^2M(z,u,\pt{x})^2 + zu\frac{uM(z,u,\pt{x})-M(z,1,\pt{x})}{u-1}\\
&+x_0Q_0(z,u,M(z,u,\pt{x}),M(z,1,\pt{x}))+\dots \\
&+x_iQ_i(z,u,M(z,u,\pt{x}),M(z,1,\pt{x}))
\end{align*}
where $Q_j(z,u,M(z,u,\pt{x}),M(z,1,\pt{x}))$ are generating functions which describe the situation where the root edge of the map is incident to a single pattern ($j=0$) or an intersecting pair of intersection type $j \in [i]$. The variables $\pt{x}$ are associated with labelled patterns of the various types. Subsequently, we would need to solve the equations and do a (multivariate) coefficient extraction of $[x_0^{\ell_0}x_1^{\ell_1}\cdots x_i^{\ell_i}]$, where $\ell_0 + 2\sum_{j=1}^i \ell_j = k$. Unfortunately, this would lead to a fairly technical and complicated analysis of the generating functions.

Thus, we will take a different approach and reduce the problem to computing the moments of face counts by deleting the interior edges of the marked pattern occurrences. Face counts in turn have been proven to satisfy a central limit theorem in~\cite{DrPa} and their probability generating function satisfies~(\ref{eq:power}) such that we can apply Lemma~\ref{LeMm}. The factorial moments of the pattern counting random variable then turns out to be just a linear combination of the factorial moments of face counting variables.

We will illustrate the procedure in the next section for a specific example and present a general theorem in Section \ref{sec:main}.

\section{Monitoring koalas in random planar maps}\label{sec:koalas}

This section is devoted to proving a central limit theorem of the number of occurrences of a pattern which we call \emph{koala}. Note that since we count pattern occurrences up to root-face preserving isomorphisms, there is no need to specify the root edge. A koala consists of a simple 4-gon adjacent to two simple 2-gons such that the whole pattern has a simple boundary. The root face of a koala is a simple 4-gon (see Figure \ref{fig:single}). Once the reader reaches the next section, they can think of it as a detailed example of the more general theorem given there. However, we hope it clarifies the necessity of several additional parameters and variables in the proof of the general central limit theorem for patterns with simple boundary.

\begin{figure}
    \centering
    \begin{subfigure}{0.48\textwidth}
        \includegraphics[width=\textwidth, page=2]{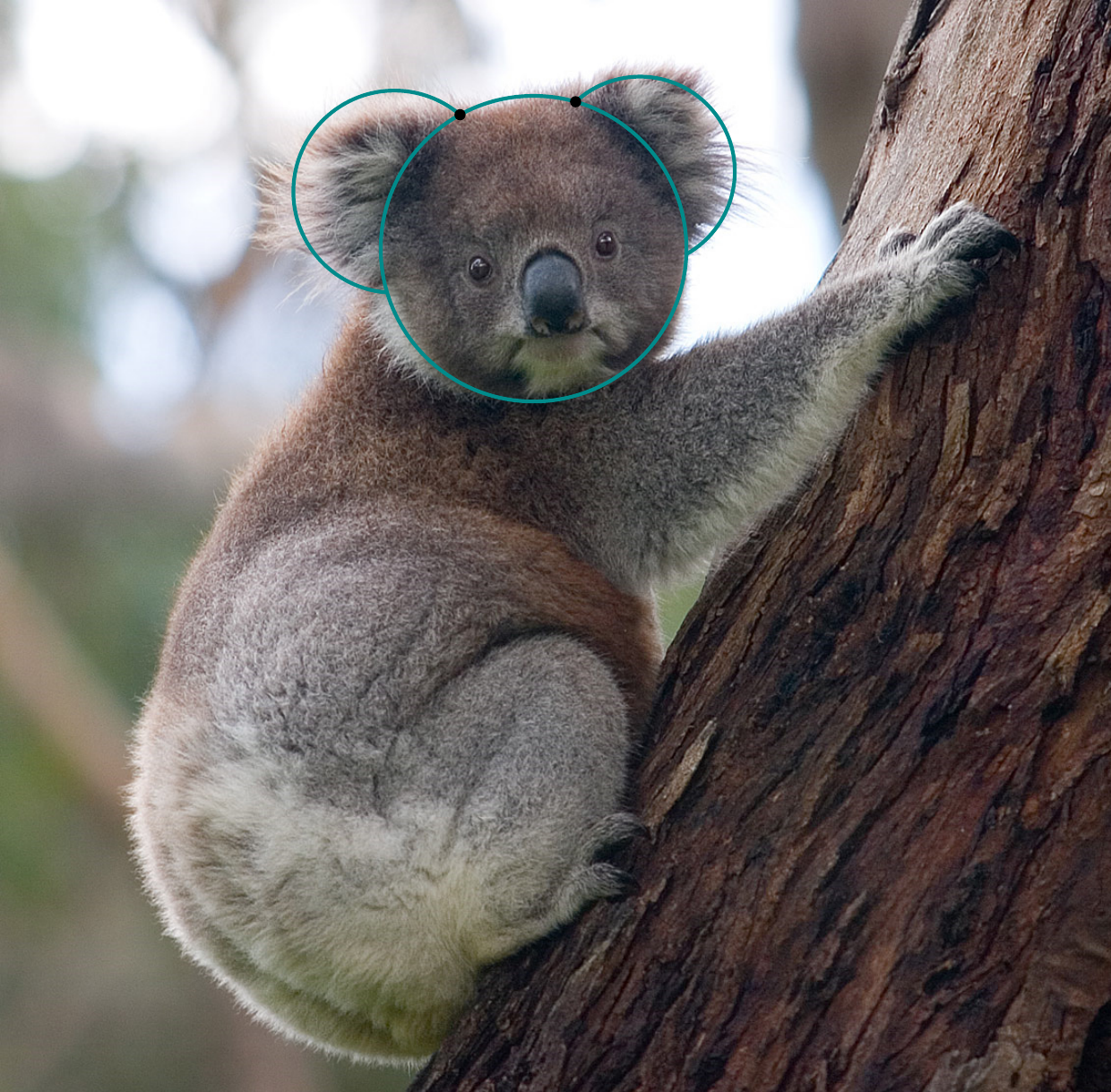}
        \caption{Photo by DAVID ILIFF. (CC BY-SA 3.0)}
    \end{subfigure}\hfill
    \begin{subfigure}{0.48\textwidth}
        \includegraphics[width=\textwidth, page=4]{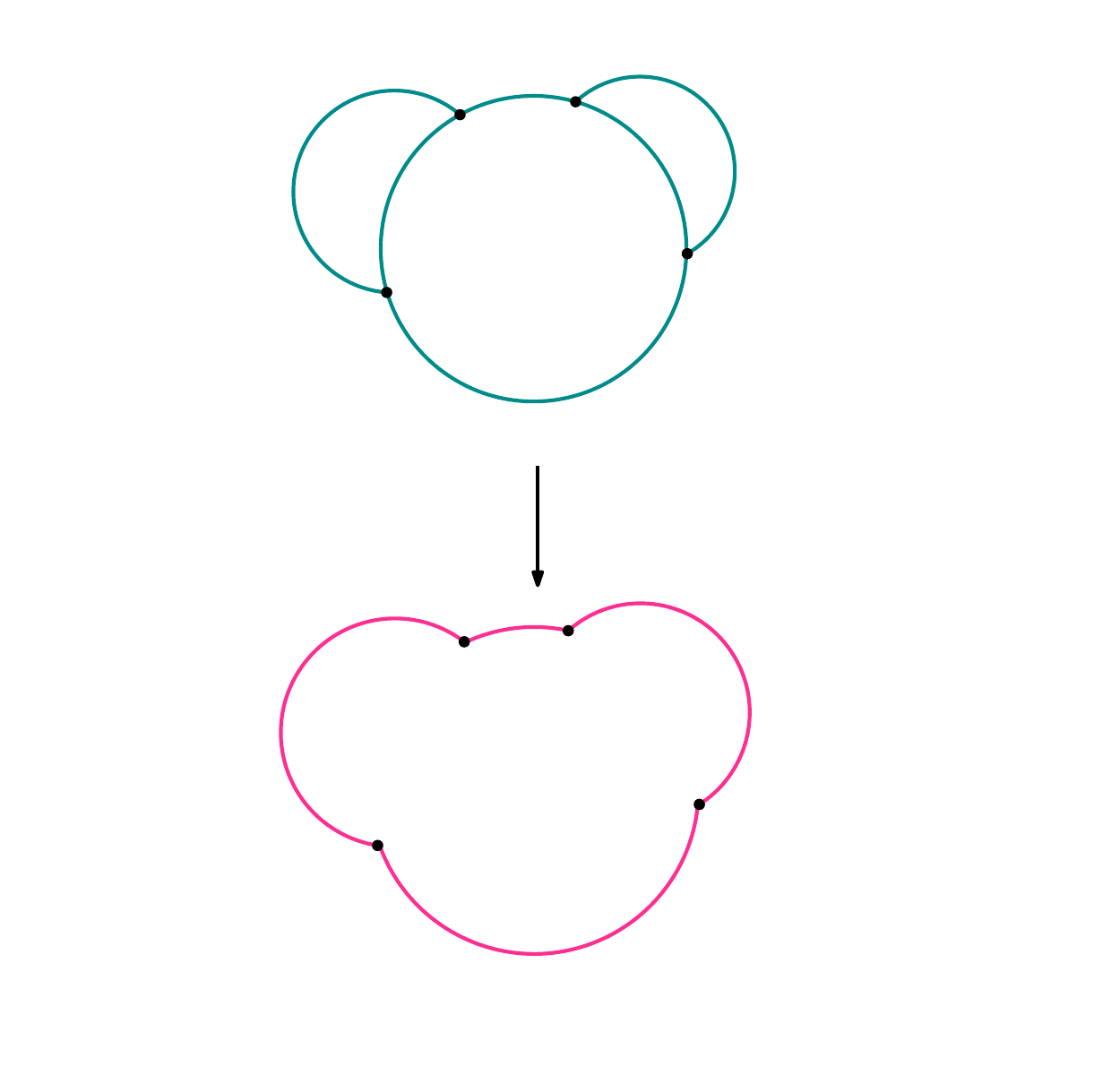}
        \caption{Koala patterns have a simple boundary}
    \end{subfigure}
    \caption{A koala (pattern).}
    \label{fig:single}
\end{figure}

\begin{theo}\label{thm:central_koala}
    Let $X_n$ be the number of koalas in a random rooted planar map with $n$ edges. Then
    \[
        \frac{X_n - \mu_n}{\sigma_n} \;\longrightarrow\; \mathcal{N}(0,1)
    \]
    where $\E[X_n] \sim \mu n$, $\V[X_n] \sim \sigma^2n$ and $\mu, \sigma \in \R$.
\end{theo}
Expanding on our comment in Section 1, the classical approach to counting maps with a given number of koala patterns would define counting variables one for each way that occurrences can intersect. Since they can intersect in arbitrarily large overlapping clusters, this would require an unbounded number of counting variables. However, Theorem \ref{thm:gao_wormald} offers a workaround as, to compute the factorial moments of order $k = \Theta(\mu_n/\sigma_n)$ we only need to compute the number of maps with $k$ selected pattern occurrences which are labelled from $1$ to $k$. The advantage of this approach is that, although the selected pattern occurrences may still overlap in large clusters, it turns out that we may ignore the cases where three or more of the $k$ selected occurrences form a cluster. Our general strategy is to reduce the problem of counting pattern occurrences to a face count problem. For example, the first factorial moment, that is the expectation, equals
\[
    \E[X_n] = \frac{m_{n,1}^\circ}{m_n}
\]
where $m_{n,1}^\circ$ is the number of maps with $n$ edges and one marked koala occurrence. This number in turn is equal to twice the number of maps with $n-2$ edges and one marked and labelled simple $4$-gon since we can reinsert the two deleted interior edges in two distinct ways (see Figure~\ref{fig:exp}). We denote the number of maps with $n-2$ edges and one marked and labelled simple $4$-gon by $\tilde{m}^\circ_{n-2,1}$. 
It follows that
\[
    \E[X_n] = \frac{m_{n,1}^\circ}{m_n} = \frac{2\tilde{m}_{n-2,1}^\circ}{m_n}.
\]
In~\cite{Yu} generating functions with variables counting simple faces have been established and in~\cite{DrmotaNoyYu} they have been proven to satisfy a central limit theorem. In particular, it is a direct consequence that the probability generating function of simple face counts satisfies~(\ref{eq:power}) such that we can compute the expectation which equals $\frac{\tilde{m}_{n-2,1}^\circ}{m_{n-2}}$ by the given formulas~(\ref{eq:exp_var_uni}). 

For the $k$-th factorial moments, the computations are slightly more involved. The general strategy to reduce the problem to counting faces is the same as for the expectation though.
\begin{figure}
    \centering
    \includegraphics[width = 0.9\textwidth]{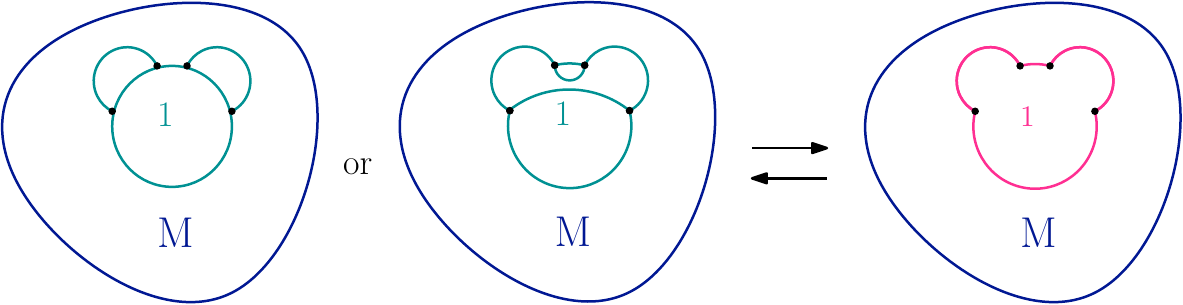}
    \caption{Reducing a koala count to a face count}
    \label{fig:exp}
\end{figure}

We first compute exact asymptotics for $\mu_n$ and $\sigma_n$ in this way so that we can conclude $\mu_n/\sigma_n = \Theta(\sqrt{n})$. It then follows by Lemma \ref{lem:S1} that the $k$-th factorial moment satisfies
\begin{equation}\label{eq:kthmom2}
    \frac{m_{n,k}^{\circ,\times} }{m_{n}}\quad \leq \quad \E\left[(X_n)_k\right] = \frac{m^{\circ}_{n,k}}{m_n} \quad \leq \quad   \frac{m_{n,k}^{\circ,\times}}{m_{n}}+\left(\frac{\mu_n}{2}\right)^k + o\Big( \E\left[(X_n)_k\right]\Big)
\end{equation}
where $m^\circ_{n,k}$ is the number of maps with $n$ edges and $k$ selected pattern occurrences (among arbitrarily many occurrences) which are labelled from $1$ to $k$  and $m^{\circ,\times}_{n,k}$ is the number of all such maps with at least $\mu_n/2$ pattern occurrences in total and where each labelled pattern occurrence intersects at most one other labelled occurrence of a koala. Since $m^{\circ,\times}_{n,k}$ includes the number of maps where no labelled koala intersects another labelled koala, which  asymptotically equals $(\mu_n)^km_n$, the main asymptotics of $\E\left[(X_n)_k\right]$ with $k = \Theta(\sqrt{n})$ as $n\rightarrow \infty$ are the same as $\frac{m_{n,k}^{\circ,\times} }{m_{n}}$. So our main focus is on labelled pattern occurrences that intersect at most pairwise.

The first step to compute $m_{n,k}^{\circ,\times}$ for $k\geq 2$ is to list all possible ways that two pattern occurrences can intersect (see Figure \ref{fig:intersects}). 

\begin{figure}
    \centering
    \includegraphics[width = 0.9\textwidth, page=6]{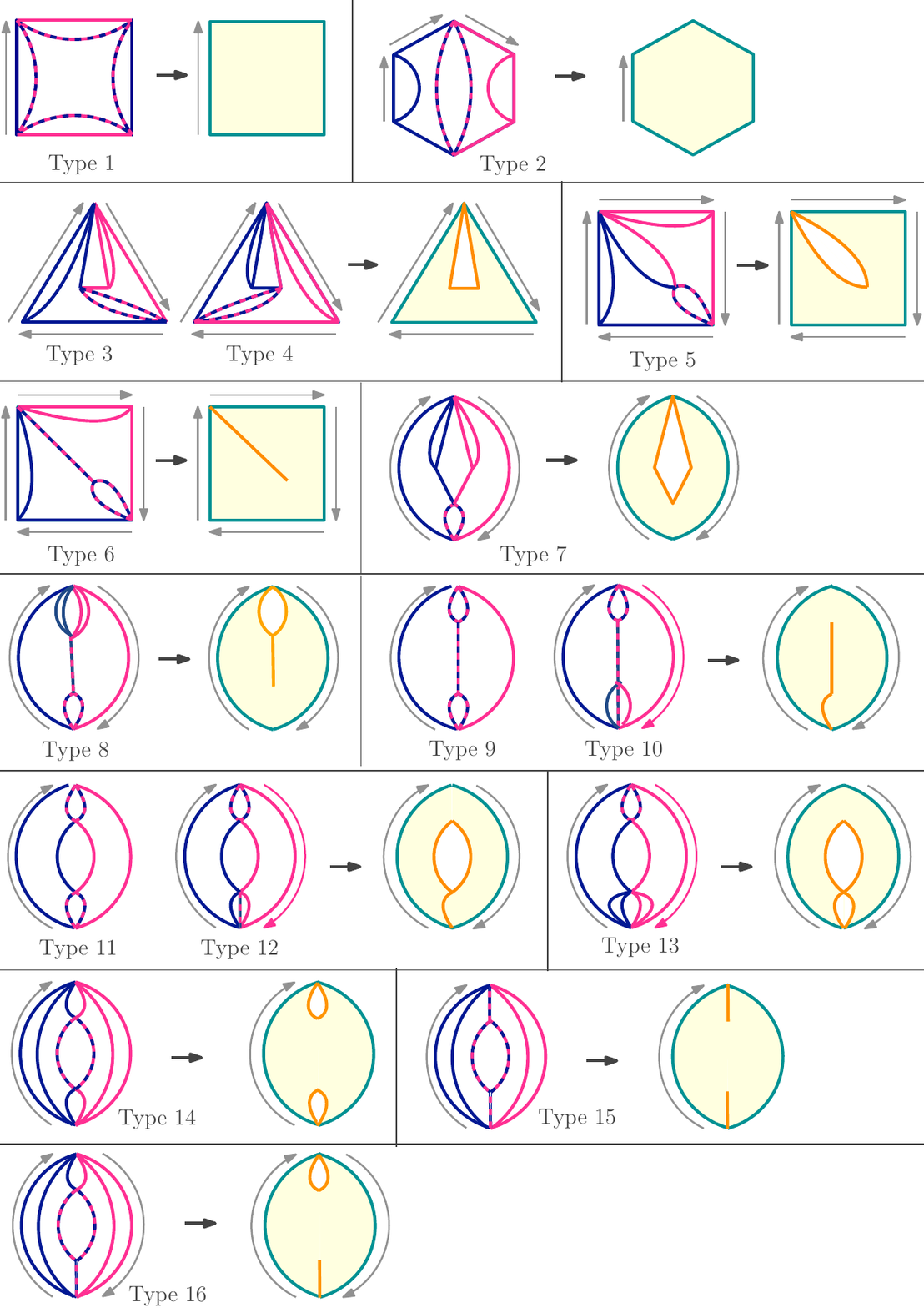}
    \caption{The intersection types of intersecting pairs of koalas. The edges of one koala are blue, the edges of the other koala occurrence are pink. For each one, the  various rooted intersection types are indicated by arrows. After deletion of interior edges, the interior faces of the two occurrences have merged into one face shaded yellow, and there are possibly some other interior faces.  The arrows show the  various rotations of this map.}
    \label{fig:intersects} 
\end{figure}

\begin{df}[Intersection types, rotations and proper embeddings]    
   Two pattern occurrences {\em intersect} if they contain at least one face in common.  Let $\map{p}$ and $\map{r}$ be rooted planar maps and let $\map{r}$ contain two distinguished 
   pattern occurrences $p_1$ and $p_2$ of $\map{p}$ which intersect. Further, suppose that each edge and vertex of $\map{r}$ is an edge or vertex of $p_1$ or $p_2$. Then $\map{r}$ is called a \emph{rooted intersection type} of the pattern $\map{p}$.  
   
\noindent    If  a rooted  map  $\map{n}$ can be obtained by rerooting $\map{r}$ such that the root face is preserved, then $\map{n}$ is called a {\em rotation} of $\map{r}$. 
    
\noindent    An \emph{intersection type} $i$ of $\map{p}$ is the set of all rotations of a rooted intersection type $\map{r}$ of $\map{p}$. We further denote the number of pairwise non-isomorphic rotations of $\map{r}$ by $r_i$, where isomorphisms must preserve $\{p_1,p_2\}$. That is, the intersection type $i$ contains $r_i$ distinct rooted maps with an unordered pair of occurences of $\map{p}$ distinguished. 

\noindent Given two intersecting  occurrences $p_1$ and $p_2$  of the pattern $\map{p}$ in an arbitrary map $\map{m}$, we can obtain a rooted intersection type of $\map{p}$ by deleting all edges of $\map{m}$ except those in $\map{p}_1$ or  $\map{p}_2$, and assigning a root on the boundary of the (new) face that contains the root face of $\map{m}$. This essentially defines an embedding of the rooted intersection type into $\map{m}$. The intersection type does not depend on which root was chosen on the boundary. If it is type $i$, we say that {\em the intersecting pair of occurrences, $p_1$ and $p_2$,  has intersection type~$i$}.
Note  that  a rooted intersection type can have interior faces that are not interior to either of the two occurrences of $\map{p}$. We call these {\em deep} faces. By a {\em proper embedding} of a rooted intersection type $\map{r}$ in a map $\map{m}$, we mean a subset of the vertices and edges of $\map{m}$ with a designated root, that induce a rooted map isomorphic to $\map{r}$, such that only its root face and the deep faces contain any other vertices or edges of $\map{m}$. 
\end{df}
\smallskip

The koala has sixteen different intersection types as illustrated in Figure~\ref{fig:intersects}. The arrows indicate possible rootings of the intersection type. Type 1 contains one rooted intersection type, Type 2 contains three different rooted intersections types, and so on.   

For any intersecting pair of koala occurrences in a given map, we aim to delete the interior edges.
Depending on the intersection type however, the deletion of \emph{all} of the interior edges can disconnect the map. This occurs only in Types 9 and 11 in Figure \ref{fig:intersects}.
Technically, our proof could be adapted to cope with this situation by using generating functions for the disconnected parts. However, the analysis of the generating functions would become more technical and less intuitive.
We find it much more convenient to keep the map connected by retaining a minimal number of interior edges, depending on the intersection type. Consequently, we define that for intersection types 9 and 11 we retain  one interior edge to reconnect the disconnected edge (and $S_2$-map, respectively) with the boundary. Figure \ref{fig:intersects} uniquely defines  the edges to be deleted for each rooted intersection type. The precise deletion process is defined below. When these deletions are applied to a rooted intersection type, which has no edges inside its deep faces, we call the resulting rooted map the {\em post-deletion  map}.  The interior faces of the intersecting patterns, together with the deleted edges, have merged to form a face (shaded with yellow in the figure) which we call the {\em deletion face} of the post-deletion map. 

Proper embeddings of post-deletion maps are defined analogously to those of rooted intersection types; again, the deep faces and the root face of the embedded map may contain other edges and vertices.

In Figure~\ref{fig:intersects}, the green arrows indicate the root edges of   various   rotations of the   post-deletion map. The fact that the number of these does not  necessarily coincide with the number of rotations of the original intersection type  will be important in the counting procedure described below. 

Our approach requires counting the maps that are produced after the above-mentioned deletions, and for this we introduce additional variables $x_1,x_2,\dots,x_{13}$ which mark the different types of faces  which can result from deleting those interior edges. Note that the type of the deletion face does not depend on which particular rotation of the rooted intersection type is used.  Here the `type' of a face refers to the graph structure of its boundary, which in particular determines its valency. However, some types of faces can be produced from more than one intersection type after the deletions; for instance,   both intersection types 9 and 10   produce the same type of face. In all, we have nine types of face to deal with. To economise on the number of variables, we define a surjective function $t: [16] \rightarrow [13]$  such that intersection type $i$ creates a face of type $t(i)$ under the above deletion process. The definition of $t$ is implied by Figure~\ref{fig:intersects}; for example $t(9)=t(10)=8$.

For our argument, we will apply the deletions to a proper embedding of a rooted intersection type while retaining the contents of its deep faces (and root face). To make things precise, and also for later use with arbitrary patterns, we make the following definition.  

\begin{df}[Deletion procedure.]\label{def:deleting} 
 For each intersection type of a pattern $\map{p}$, we define a canonical member, i.e.\ a rooted intersection type $\map{r}$, and define a set   $D(\map{r})$ of edges, such that each is an interior edge  of one of the two intersecting occurrences of $\map{p}$ in $\map{r}$, and such that the  deletion of $D(\map{r})$ retains connectedness of the map, and $D(\map{r})$ is  maximal in this sense. The first step of the deletion procedure consists of deleting the edges in  $D(\map{r})$ (i.e.\ merging them with their incident faces), to form a new rooted map $\map{r}'$. 
The second step of the procedure can  be performed in $c$ different ways, where $c$ is the number of nonisomorphic rotations of $\map{r}'$. One alternative is to do nothing in the second step, and the others consist of moving the root edge of $\map{r}'$ to any of the next $c-1$ possible locations around the root face, working clockwise.  This finishes the definition for the canonical rooted intersection type $\map{r}$. To define the procedure for the    other  rotations  of $\map{r}$, for $1\le i\le r-1$ let $\map{r}_i$ denote the rotation obtained by moving the root   of $\map{r}$ just $i$ positions clockwise around the root face.  The deletion procedure first moves the root of  $\map{r}_i$  just $i$ positions counter-clockwise around its root face (which produces a copy of $\map{r}$) and then applies the two steps described above. Thus, these $r$ different rotations produce exactly the same set of maps, as in Figure~\ref{FigX}. (Here $\map{r}'$ has only one rotation.)
\begin{figure}
\includegraphics[width=\textwidth]{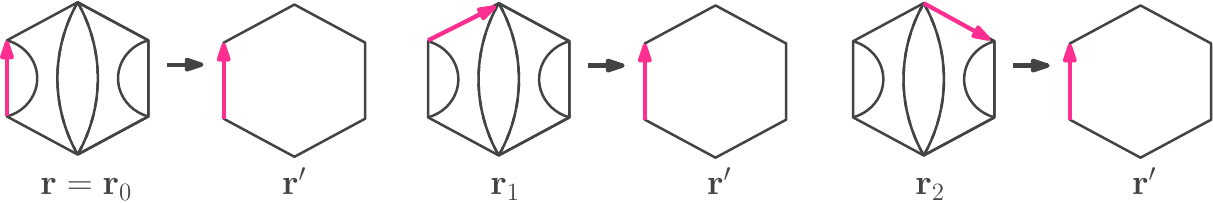}
\caption{The deletion procedure applied to these three rotations $\map{r}_0,\map{r}_1,\map{r}_2$ of a map $\map{r}$ produces exactly the same map in each case.}
\label{FigX}
\end{figure}

\end{df}

With the appropriate selection of $\map{r}$ and $D(\map{r})$, the deletion procedure for all the different types of intersecting koalas is illustrated in Figure~\ref{fig:intersects}. For each type, the rooted intersection types which are represented by an arrow on the left can produce the post-deletion map on the right with any of the indicated rootings.

Let $c_{t(i)}$ denote the number of rotations of a post-deletion map for an intersection pattern of type $i$.
The numbers $r_i$ and $c_{t(i)}$ do not always coincide, because deleting edges may create or destroy symmetries. For example, in Figure~\ref{fig:intersects}, the maps associated to the rotations in the intersection types are drawn in green and yellow. For intersection type $2$, the number of rotations is $r_2 = 3$ and the number of rotations of the associated map is just $c_{t(2)} = 1$.
A full list of rotations of each intersection type and the  map associated to it is given in Table~\ref{tab:intersects}. We need to keep track of the number of rotations  for our  counting procedure. 
  
A rooted koala has two rotations, but for the basic counting we identify these as one {\em koala occurrence}.  Consider a map $\map{m}$ with $k$ selected koala occurrences   which are labelled from $1$ to $k$ and intersect at most pairwise. We
\begin{enumerate}
    \item unlabel the $k$ selected koalas, 
    \item delete their interior edges according to Definition~\ref{def:deleting} and
    \item distinguish the resulting post-deletion faces.
\end{enumerate} 
We denote the map so obtained by $\tilde{\map{m}}$. Our job is made more difficult by the fact that when we count maps with distinguished faces,  $\map{m}$ will not be the only map that results in $\tilde{\map{m}}$ by this procedure. The overcounting is by a simple factor of 2 in the case of an isolated koala, but the case of intersecting koalas is more complicated. The next result quantifies this overcounting for a single intersecting pair of koalas. Let $[I]=\{1,\ldots, I\}$ denote the set of intersection types.

For later use, we state and proof the result for patterns in general.

\begin{lemma}\label{prop:fact}
    Given any pattern $m_i$ be the number of rooted maps with $n$ edges and two distinguished pattern occurrences with intersection type $i \in [I]$. Further, let $d_i$ be the number of edges which we delete in intersection type $i$ and $\tilde{m}_{i}$ be the number of maps with $n-d_i$ edges and with one distinguished face associated to intersection type $i$. Then
    \[
        m_i = \frac{r_i}{c_{t(i)}}\tilde{m}_{i},
    \]
    where $r_i$ is the number of rotations of intersection type $i$ and $c_{t(i)}$ the number rotations of each post-deletion map for intersection type $i$. 
    \end{lemma}
\begin{proof}
    To facilitate counting, we consider {\em doubly rooted} maps, which are rooted maps having a secondary root specified, besides the root of the map.
  
    Let $A_n(i)$ be the set of all   doubly rooted     $n$-edged maps  $\map{m}$ in which the secondary root is the root of a proper embedding  of a rooted intersection of type $i$ 
    in  $\map{m}$. 
    Let  $B_n(j)$ denote the set of all doubly rooted $n$-edged maps in which the secondary root is the  root  of a proper embedding  of a post-deletion map in which the deletion face is of type $j$.  The deletion procedure can be applied to  each element of $A_n(i)$ by applying it to the rooted intersecting pair having the secondary root. Given the root edge of an intersection map its embedding is uniquely determined by the incidences of the interior faces with the edges of the pattern because its boundary is simple. The second step of this procedure moves the  secondary root to any one of $c_{t(i)}$ positions, each giving a different element  of $B_n(t(i))$, and the number of elements of  $A_n(i)$ taken to one of $B_n(t(i))$ is $r(i)$. 
    Hence the total number of ways that the deletion procedure can be performed to members of  $A_n(i)$ is 
    $$ c_{t(i)}|A_n(i)|= r(i)|B_n(t(i))|.
    $$
    If $k$ is the valency of the root face of a rooted intersection map of type $i$, then clearly $|A_n(i)|=km_i$ since there are $k$ ways to apply a secondary root to the intersection type occurrence. On the other hand, for similar reasons, $B_n(t(i))=k\tilde m_i$. The lemma follows.
\end{proof}

We return to the koala example. A more involved example of overlapping patterns would be illustrated in Figure~\ref{fig:count}, where we want to determine the number $M$ of maps where $5$ koalas are selected and labelled and one of these koalas does not intersect with any other selected koalas, two form intersection type 9, the other two have intersection type 10. According to the counting procedure we unlabel the koala occurrences and delete the interior edges as defined in Figure~\ref{fig:intersects}. The result is a rooted map with one distinguished simple $4$-gon and two distinguished simple $2$-gons with a path of length two attached at one of its vertices.
In order to keep track of how many maps with a distinguished single koala, a distinguished intersection type 9 and a distinguished intersection type 10 reduce to the same map with these three distinguished faces, we go step by step. First, each of the $2$-gons with the paths attached at a vertex may rechosen as resulting from an intersection type 10. Once we choose this face we can apply Proposition~\ref{prop:fact} and conclude there are $r_{10}/c_{t(10)}$ maps with one distinguished simple $4$-gon, one distinguished simple $2$-gon with an attached path of length two and one distinguished intersection type $10$ which reduce to this map. Another application of Proposition~\ref{prop:fact} yields that there are $r_9/c_{t(9)}$ maps with one marked simple $4$-gon, a marked intersection type 9 and a marked intersection type 10 which reduce to one of the previous maps. And finally there are two ways to reinsert the interior edges of the single koala.
Therefore we have
\[
    \frac{M}{5!} = \binom{2}{1}\cdot2\cdot\frac{r_9}{c_{t(9)}}\cdot \frac{r_{10}}{c_{t(10)}}\cdot m_{n,1,0,0,0,0,0,0,2,0} = 2\cdot m_{n,1,0,0,0,0,0,0,2,0} 
\]
where $m_{n,1,0,0,0,0,0,0,2,0}$ is the number of maps with one distinguished simple $4$-gon and two distinguished simple $2$-gons with a path of length two attached at one of its vertices.

\begin{figure}[t]
 \centering
    \includegraphics[width=0.32\textwidth, page=2]{count_map.pdf} \hspace{2cm}\includegraphics[width=0.32\textwidth, page=1]{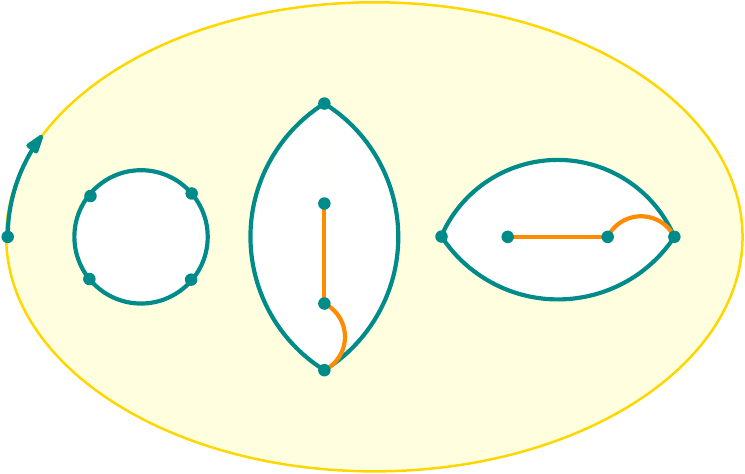}
    \caption{\textbf{Reducing pattern counts to face counts.}\\ We unlabel the five patterns and delete interior edges according to Definition~\ref{def:deleting}.}
    \label{fig:count}
\end{figure}

\subsection{Generating functions of koalas} 

Now, the main focus concerns the counting of maps with distinguished faces via generating functions. These distinguished faces are characterised by their boundary.
\begin{df}[Boundary shape]\label{def:bound}
    Let $F$ be a face of valency $\ell$, root vertex $v_1$ and root edge $(v_1,v_2)$. Its boundary is given by a sequence of vertices $(v_1,v_2,\dots, v_1)$ and can be decomposed into the boundary of a simple $h$-gon $(v_1,v_2,\dots,v_h, v_1)$ and boundaries of (possibly trivial) maps $\map{m}_1$, $\map{m}_2, \dots, \map{m}_h$ which are attached at the vertices $v_1,v_2,\dots, v_h$ respectively. The boundaries of the maps $\map{m}_i, i=1,\dots h$ in turn can be decomposed into $e_i$ bridges and boundaries of simple maps. 
    
    We define the \emph{boundary shape} of $F$ to be the vector $(h, e, s_1, s_2, s_3, \dots)$, where $e=\sum_{i=1}^h e_i$ and $s_i$ is the total number of boundaries of $S_i$-maps appearing in the boundary decomposition above.
\end{df}

For example, the boundary of the distinguished face in intersection type 7 of koalas is given by $(v_1,v_2,v_3,v_4,v_5,v_2,v_1)$ and it decomposes into the boundary of the simple $2$-gon $(v_1,v_2,v_1)$ and the boundary of an $S_4$ map dangling off $v_2$, which is given by $(v_2,v_3,v_4,v_5,v_2)$.
The boundary shapes of the distinguished faces of all rotations in an intersection type are identical. They are given in Table~\ref{tab:intersects}.
\begin{table}
    \begin{tabular}{c | c | c | c }
        Intersection type & Boundary shape &  Pattern rotations & Face rotations \\
        \hline 
         1 & $(4,0,0,0,\dots)$ & $r_1=1$ &  $c_{t(1)} = 1$\\
         2 & $(6,0,0,0,0,0,0,\dots)$ &  $r_2=3$  & $c_{t(2)} = 1$\\
         3& $(3,0,0,0,1,0,\dots)$ &  $r_3=3$ &  $c_{t(3)} = 3$\\
         4& $(3,0,0,0,1,0,\dots)$  & $r_4=3$ &  $c_{t(4)} = 3$\\
         5 & $(4,0,0,0,1,0,\dots)$  & $r_5=4$ &  $c_{t(5)}=4$\\
         6 & $(4,1,0,\dots)$  & $r_6=4$ &  $c_{t(6)} = 4$\\
         7 & $(2,0,0,0,0,1,0,\dots)$  & $r_7=2$ &  $c_{t(7)} = 2$\\
         8 &$(2,1,0,1,0\dots)$  & $r_8=2$  & $c_{t(8)} = 2$\\
         9 &$(2,2,0,\dots)$  & $r_9=1$ &  $c_9 = 2$
         \\
         10 &$(2,2,0,\dots)$  & $r_{10}=2$ &  $c_{t(10)} = 2$
         \\
         11 &$(2,1,0,1,0,\dots)$  & $r_{11}=1$ &  $c_{t(11)} = 2$
         \\
         12 &$(2,1,0,1,0,\dots)$  & $r_{12}=2$ &  $c_{t(12)} = 2$
         \\
         13 &$(2,0,0,2,0,\dots)$  & $r_{13}=2$ &  $c_{t(13)} = 2$
         \\
         14 &$(2,0,0,2,0,\dots)$  & $r_{14}=1$ &  $c_{t(14)} = 1$
         \\
         15 &$(2,2,0,0,0,\dots)$  & $r_{15}=1$ &  $c_{t(15)} = 1$
         \\
         16 &$(2,1,0,1,0,\dots)$  & $r_{16}=2$ &  $c_{t(16)} = 2$.
    \end{tabular}
    \caption{Boundary shapes, rotations and associated faces to the intersection types of koalas}\label{tab:intersects}
\end{table}   

Once we have identified the boundary shapes, the number of rotations and associated maps of all intersection types, we can set up a functional equation describing the generating function of maps with additional variables which count the associated distinguished faces. 
Subsequently, we use analytic tools to determine the asymptotic growth of the desired number of maps.

\begin{obs}\label{obs:shape} It is enough to know the boundary shape $(h, e, s_{1}, s_{2}, s_{3},\dots)$ and the number $c$ of rotations of the associated face of an intersection type to describe the generating function of rooted planar maps where the root edge is incident to a face associated to intersection type in terms of the generating functions $P_i$ and $S_i$.

In particular, their generating function equals
\[
    czu^{2-h}P_{h-1}(z,u)z^{e}S_1(z)^{s_{1}}S_2(z)^{s_{2}}S_3(z)^{s_{3}}\cdots,
\]
where $P_i(z,u)$ is the generating function of $P_i$-maps and $S_i(z)$ is the generating function of $S_i$-maps.\end{obs} 
\begin{proof}
The decomposition process is analogous to the proof of Theorem~\ref{prop:nonself} and is also illustrated in Figure~\ref{fig:map_face}. If the root edge is incident to a face of boundary shape $(h, e, s_{1}, s_{2}, s_{3},\dots)$, then there are $c$ rotations of this face that could be incident to the root. By Definition~\ref{def:bound}, the boundary of the face consists of a pure $h$-gon that decomposes into the root edge and the $h-1$ edges on the boundary of a $P_{h-1}$-map and the boundary of maps which are dangling off its vertices. These maps further decompose into $s_i$ distinct $S_i$-maps, $i\geq 1$ and $e$ bridges. Combining the generating functions of all of these components, we end up with the described generating function.
\end{proof}

Analogously, we can obtain the formulas for generating functions which contain variables counting faces associated to intersection types. Let
\begin{align*}
    M(z,u,\pt{x}) &= \sum_{n,k,t_1,t_2\dots t_{13} \geq 0} m_{n,k,t_1,t_2\dots t_{13}}z^nu^kx_1^{t_1}x_2^{t_2}\cdots x_{13}^{t_{13}}\\
    M(z,1,\pt{x}) &= \sum_{n,t_1,t_2\dots t_{13} \geq 0} m_{n,t_1,t_2\dots t_{13}}z^nx_1^{t_1}x_2^{t_2}\cdots x_{13}^{t_{13}}
\end{align*}
be the generating functions of rooted planar maps with additional variables $\pt{x}=(x_1, x_2, \dots, x_{13})$, where $x_1$ counts simple $4$-gons which are associated to single patterns and for each $i\in [16]$, the variable $x_{t(i)}$ counts faces associated to intersection type $i$.

By Observation~\ref{obs:shape}, we can decompose a planar map as in Figure \ref{fig:map_face} and derive the functional equation
\begin{align}\label{sys:koalas}
    M(z,u, \pt{x}) = 1&+zu^2M(z,u,\pt{x})^ 2+zu\frac{uM(z,u,\pt{x})-M(z,1,\pt{x})}{u-1}\\
    & +(x_1-1) \frac{z}{u^2}P_3(z,u,\pt{x})
    +(x_2-1) \frac{z}{u^4}P_5(z,u,\pt{x})
    \nonumber\\
    & + (x_3-1) \frac{3z}{u}P_2(z,u,\pt{x})S_3(z,\pt{x})
    +(x_4-1) \frac{4z}{u^2}P_3(z,u,\pt{x})S_2(z,\pt{x}) \nonumber\\
    & + (x_5-1) \frac{4z^2}{u^2}P_3(z,u,\pt{x}) 
    +(x_6-1) 2zP_1(z,u,\pt{x})S_4(z,\pt{x})
    \nonumber\\
    & + (x_7-1) 2z^2P_1(z,u,\pt{x})S_2(z,\pt{x})
    +(x_8-1) 2z^3P_1(z,u,\pt{x})\nonumber\\
    &+ (x_9-1) 2z^2P_1(z,u,\pt{x})S_2(z,\pt{x}) + (x_{10}-1) 2zP_1(z,u,\pt{x})S_2(z,\pt{x})^2\nonumber\\
    &+ (x_{11}-1) zP_1(z,u,\pt{x})S_2(z,\pt{x})^2 + (x_{12}-1) z^3P_1(z,u,\pt{x})\nonumber\\
    &+ (x_{13}-1) 2z^2P_1(z,u,\pt{x})S_2(z,\pt{x}),
\end{align}
where $S_i(z,\pt{x})$ denotes the generating function of maps with simple boundary and root face valency~$i$ and the generating functions $P_k(z,u,\pt{x})$ count the number of maps with a partial simple boundary of length~$k$ (that is, $S_i$-maps and $P_i$-maps respectively). 
\begin{figure}
    \centering
    \includegraphics[width=0.8\textwidth]{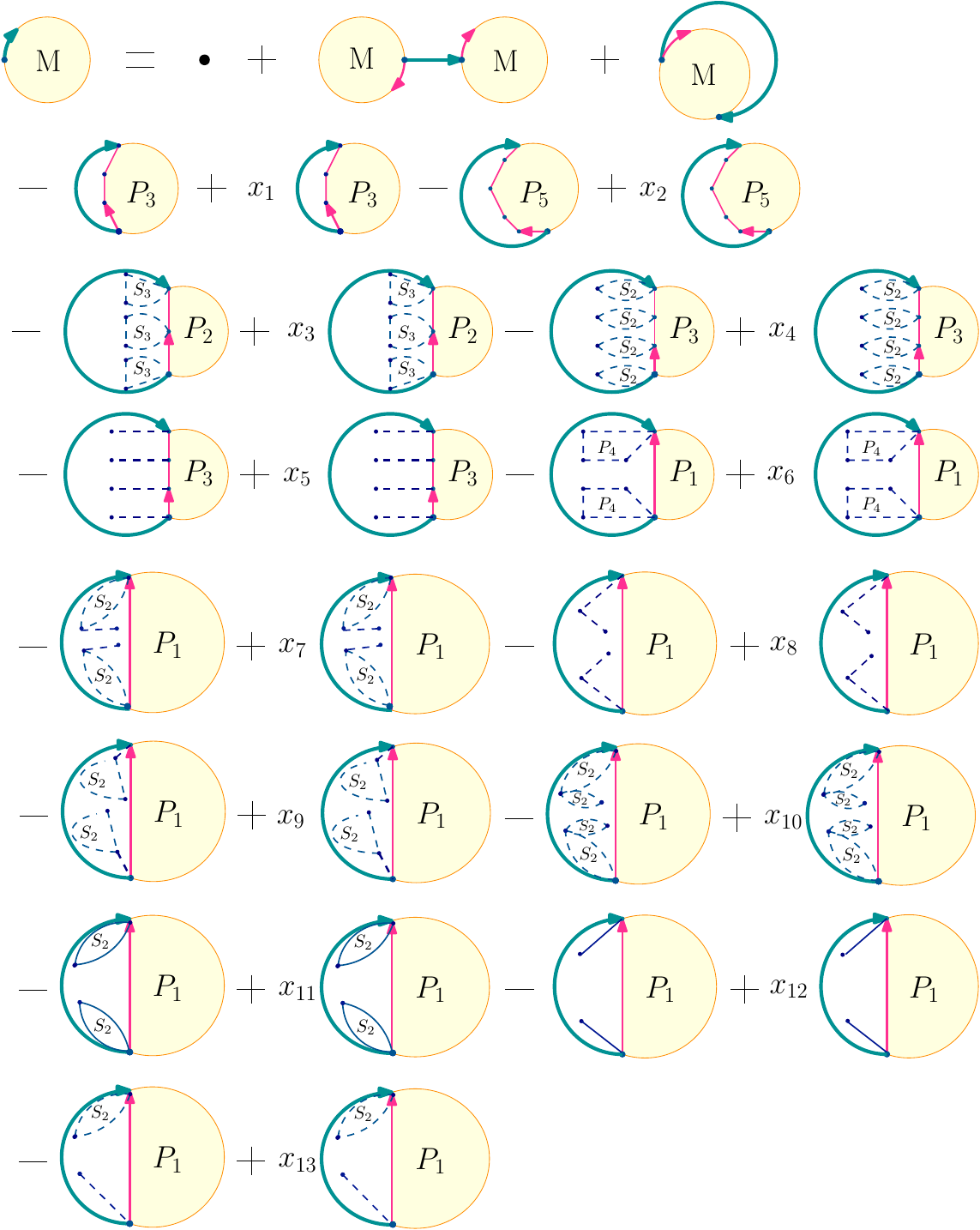}
    \caption{Decomposition of rooted planar maps with counting variables for faces associated to the intersection types of koalas}
    \label{fig:map_face}
\end{figure}

The generating functions for $S_i$-maps and $P_i$-maps have already been covered in Section~\ref{sec:intro} and we already observed that the systems for $P_\ell(z,u,\pt{x})$ and $S_\ell(z,\pt{x})$ are iterative such that we can substitute the generating functions in the main equation~(\ref{sys:koalas}) by polynomials in $M(z,u,\pt{x}), M(z,1,\pt{x})$ and $(m_i(z,\pt{x}))_{0\leq i\leq \ell}$. In turn, the series $m_i(z,\pt{x})$ can be computed by evaluating the first~$i$ coefficients in $u$ in the main equation for $M(z,u,\pt{x})$.

The main observation is that all $m_i(z,\pt{x})$ can be expressed as analytic functions in $M(z,1,\pt{x})$ and additional functions which stay analytic at the singularity $r(\pt{x})$ (see also~\cite{DrPa}). The reduced system then falls into the regime of Theorem~3 in~\cite{DrmotaNoyYu} which yields a solution of the form
\[
    M(z,1,\pt{x}) = G(z,\pt{x}) + H(z,\pt{x})\left(1-\frac{z}{r(\pt{x})}\right)^{\frac{3}{2}}
\]
where $G(z,\pt{x})$ and $H(z,\pt{x})$ are analytic functions in a neighbourhood of $(r(\pt{1}),\pt{1})$. 
Lemma \ref{LeMm} applies then to the corresponding probability generating function $[z^n]M(z,1,\pt{x})/ [z^n]M(z,1,\pt{1})$.

\subsection{Expectation and variance of koala occurrences}
We start with the first moment, where we count maps with a single marked pattern. After deletion of the interior edges we are thus interested in the number $m_{n-2,1,0,0,0,0,0,0,0,0}$ of  maps on $n-2$ edges with one marked simple 4-gon. 
Clearly,
\[
    m_{n-2,1,0,0,0,0,0,0,0,0} = [z^{n-2}] \partial_{x_{1}}M(z,1,\pt{x})|_{\pt{x}=\pt{1}},
\]
which we can easily determine by Lemma \ref{LeMm} since
\[
    \E \left[Y_{n-2}^{(1)} \right] = \frac{m_{n-2,1,0,0,0,0,0,0,0,0}}{m_{n-2}} = (n-2)f_{x_{1}}(\pt{1}) + g_{x_{1}}(\pt{1}) + O\left(\frac{1}{n}\right),
\]
where $\pt{Y}_n = \left(Y_n^{(1)},Y_n^{(2)},\dots, Y_n^{(8)}\right)$ is the vector of random variables $Y_i$ which give the number of faces counted by the variable $x_i$.
Now, there are two rotations of a single koala which could have been in the distinguished face. Therefore, in the final computation, we obtain
\begin{align*}
    \E(X_n) &= \frac{2 m_{n-2,1,0,0,0,0,0,0,0,0}}{m_n} = \frac{2 m_{n-2,1,0,0,0,0,0,0,0,0}}{m_{n-2}}\cdot\frac{m_{n-2}}{m_n} \\
    &= 2\left((n-2)f_{x_{1}}(\pt{1}) + g_{x_{1}}(\pt{1}) + O\left(\frac{1}{n}\right)\right)\frac{[z^{n-2}] M(z,1,\pt{1})}{[z^n] M(z,1,\pt{1})}
\end{align*}

We will need an expansion up to constant order of $\E(X_n)$ in the computation of the variance. Thus, we will be more precise in the computation of the second factor. 

\begin{lemma}\label{le:frac_map}
    Let $M(z,1,\pt{1})$ be the generating function of rooted planar maps, where $z$ marks the number of edges. Then
    \[
        \frac{[z^{n-k}] M(z,1,\pt{1})}{[z^n] M(z,1,\pt{1})}  = \left(1+\frac{5k}{2n}+ O\left(\left(\frac{k}{n}\right)^2\right)\right)12^{-k}
    \]
\end{lemma}
\begin{proof}
    Tutte \cite{tutte_1963} proved that the number of planar maps on $n$ edges is
    \[
        [z^n]M(z,1,\pt{1}) = \frac{2(2n)!}{(n+2)!n!}3^n
    \]
    and therefore
    \begin{equation*}\label{eq:exp_coeffs}
        \frac{[z^{n-k}] M(z,1,\pt{1})}{[z^n] M(z,1,\pt{1})} = \frac{(2n-2k)!(n+2)!n!}{(n-(k-2))!(n-k)!(2n)!}3^{-k} = \left(1+\frac{5k}{2n}+ O\left(\left(\frac{k}{n}\right)^2\right)\right)12^{-k}.
    \end{equation*}
\end{proof}

If we plug these asymptotics in our calculation of $\E [X_n ]$, we obtain
\begin{align*}
    \E [X_n] &= \frac{2}{12^2}\left((n-2)f_{x_{1}}(\pt{1}) + g_{x_{1}}(\pt{1}) + O\left(\frac{1}{n}\right)\right) \left(1+\frac{5}{n} + O\left(\frac{1}{n^2}\right)\right)\\
    &= \frac{2}{12^2} f_{x_{1}}(\pt{1})\left(n-2+\frac{g_{x_{1}}(\pt{1})}{f_{x_{1}}(\pt{1})} + O\left(\frac{1}{n}\right)\right) \left(1+\frac{5}{n} + O\left(\frac{1}{n^2}\right)\right)\\
    &= \frac{2}{12^2} f_{x_{1}}(\pt{1})\left(n + 3+\frac{g_{x_{1}}(\pt{1})}{f_{x_{1}}(\pt{1})}+ O\left(\frac{1}{n}\right)\right)
\end{align*}
where $f(\pt{x}) = \log(r(\pt{1}))-\log(r(\pt{x}))$, $g(\pt{x}) = \log(H(r(\pt{x}),\pt{x}))-\log(H(r(\pt{1}),\pt{1}))$ and the partial derivatives are computable constants analogous to Example~\ref{ex:dgtriang}.

We move on to the second moment where we count maps with two selected and labelled koalas as given by
\[
    \E\left[X_n(X_n-1)\right] =  \frac{m_{n,2}^\circ}{m_n} .
\]
Either these two koalas do not intersect or we find one of the twelve intersection types. 

The case where they do not intersect lead to analogous considerations as above such that we count maps with two distinguished simple $4$-gons after deletion of the four interior edges. For intersection types 9 and 11 we delete three interior edges, for all others all four. To shorten notation, we define the vector which counts the number of deleted edges in intersection type $i$ as 
\[
    d = (d_1,d_2,\dots, d_{16}) = (4,4,4,4,4,4,4,4,3,4,3,4,4,4,4,4).
\]
To employ Lemma~\ref{prop:fact}, instead of counting maps with distinguished koala occurrences we count maps with distinguished faces weighted by 
\[
    \frac{r_i}{c_{t(i)}} = 
    \begin{cases}
        1 \quad &\mbox{if } i\in \{1,3,4,5,6,7,8,10,12,13,14,15\}\\
        3 \quad &\mbox{if } i = 2\\
        \frac{1}{2} \quad &\mbox{if } i \in \{9,11\},
    \end{cases}
\]
where $r_i$ is the number of rotations of the intersection type, $t(i)$ gives the index of the variable and $c_{t(i)}$ denotes the number of rotations of the associated face.
That is,
\[     
    (t(1),t(2),\dots, t(16)) = (1,2,3,3,4,5,6,7,8,8,9,9,10,11,12,13).
\]
and we ultimately obtain
\[
    \frac{1}{2!} \frac{m_{n,2}^\circ}{m_n} = 2^2\, \frac{m_{n-2\cdot 2,2,0,0,0,0,0,0}}{m_n} + \sum_{i=1}^{16} \left(\frac{r_i}{c_{t(i)}}\; \frac{m_{n-d_i,e_{t(i)}}}{m_n}\right),
\]
where $e_j$ denotes the $j$-th unit vector. 
Analogous to the computation of the expectation, we can determine these numbers by computing the factorial moments of $\left(Y_n^{(1)},Y_n^{(2)},\dots, Y_n^{(13)}\right)$ which are the random number of faces counted by the variables $x_i$ respectively. In particular, since 
\[
    \E \left[Y_{n-4}^{(1)}\left(Y_{n-4}^{(1)}-1\right)\right] = 2!\;\frac{m_{n-4,2,0,0,0,0,0}}{m_{n-4}}
\]
we have 
\begin{align*}
    \E\left[X_n(X_n-1)\right]
    &=2^2\;\frac{m_{n-4}}{m_n} \;\E \left[Y_{n-4}^{(1)}(Y_{n-4}^{(1)}-1)\right]+
    +2!\; \sum_{i=1}^{12} \frac{r_i}{c_{t(i)}}\;\frac{m_{n-d_i}}{m_n} \;\E \left[Y_{n-d_i}^{(t(i))}\right].
\end{align*}
Once again, we apply Lemmas~\ref{LeMm} and~\ref{le:frac_map} such that
\begin{align*}
    \frac{m_{n,2}^\circ}{m_n}
    &=\frac{4}{12^4}\Big(f_{x_{1}}(\pt{1})^2(n-4)^2 +\big(f_{x_1x_1}(\pt{1})+2g_{x_{1}}(\pt{1})f_{x_{1}}(\pt{1})\big)n \Big)\left(1+\frac{10}{n}+ O\left(\frac{1}{n^2}\right)\right)\\
    &\qquad + \sum_{i=1}^{16} \frac{2r_i}{12^{d_i}c_{t(i)}}f_{x_{t(i)}}(\pt{1})n
    \left(1+O\left(\frac{1}{n}\right)\right)+O\left(1\right)\\
    &=\frac{4}{12^4}\Big(f_{x_{1}}(\pt{1})^2n^2 +\big(f_{x_1x_1}(\pt{1})+2g_{x_{1}}(\pt{1})f_{x_{1}}(\pt{1})+2f_{x_1}(\pt{1})^2\big)n \Big)\\
    &\qquad + \sum_{i=1}^{16} \frac{2r_i}{12^{d_i}c_{t(i)}}f_{x_{t(i)}}(\pt{1})n+O(1)
\end{align*}
Since $\mathbb{V}(X_n) = \E(X_n(X_n-1))+\E(X_n) -\E(X_n)^2$ 
we can compute the variance of $X_n$ as
\begin{align*}
    \mathbb{V}(X_n) 
    &= \frac{4}{12^4}\Big(f_{x_{1}}(\pt{1})^2n^2 +\big(f_{x_1x_1}(\pt{1})+2g_{x_{1}}(\pt{1})f_{x_{1}}(\pt{1})+2f_{x_1}(\pt{1})^2\big)n \Big)\\
    &\quad + \sum_{i=1}^{16} \frac{2r_i}{12^{d_i}c_{t(i)}}f_{x_{t(i)}}(\pt{1})n
    + \frac{2}{12^2} f_{x_{1}}(\pt{1})n \\
    &\quad - \frac{4}{12^4} f_{x_{1}}(\pt{1})^2\left(n + 3+\frac{g_{x_{1}}(\pt{1})}{f_{x_{1}}(\pt{1})}+ O\left(\frac{1}{n}\right)\right)^2 + O(1)\\
    &= \frac{4f_{x_{1}}(\pt{1})^2}{12^4}n^2 + \frac{4f_{x_{1}}(\pt{1})^2}{12^4}\left( \frac{f_{x_1x_1}(\pt{1})}{f_{x_1}(\pt{1})^2}+\frac{2g_{x_{1}}(\pt{1})}{f_{x_{1}}(\pt{1})}+2\right)n+ \sum_{i=1}^{16} \frac{2r_i}{12^{d_i}c_{t(i)}}f_{x_{t(i)}}(\pt{1})n\\
    &\quad + \frac{2}{12^2} f_{x_{1}}(\pt{1})n - \frac{4f_{x_{1}}(\pt{1})^2}{12^4} n^2-\frac{4f_{x_{1}}(\pt{1})^2}{12^4}  \left(6+\frac{2g_{x_{1}}(\pt{1})}{f_{x_{1}}(\pt{1})}\right)n+ O(1)\\
    &= \left(\frac{4}{12^4}f_{x_{1}x_{1}}(\pt{1})-\frac{16}{12^4}f_{x_{1}}(\pt{1})^2  + \sum_{i=1}^{16} \frac{2r_i}{12^{d_i}c_{t(i)}}f_{x_{t(i)}}(\pt{1})  + \frac{2}{12^2} f_{x_{1}}(\pt{1})\right)n + O(1)
\end{align*}

\subsection{The k-th moment}
If we compute the $k$-th factorial moment,  where $k = \Theta(\sqrt{n})$, we count maps with several selected and labelled koala pattern occurrences. By Lemma~\ref{lem:S1} it is enough to count only those maps where a selected koala intersects with at most one other selected koala. We therefore sum over all $\pt{k} = (s,p_1,p_2,\dots p_{12})$, where $s$ is the number of selected single koalas (which do not intersect with other koala occurrences), $p_i$ the number of selected occurrences of intersection type $i$ respectively and $s+ \sum 2 p_i= k$. When we count maps with selected and labelled koalas of intersection types $\pt{k} = (s,p_1,p_2,\dots p_{12})$, we delete in total 
\[
    D:= 2s + \sum d_ip_i = 2k-p_9-p_{11}
\] 
edges and end up with $s+p_1$ faces counted by $x_1$, $p_2$ faces counted by $x_2$, $p_3+p_4$ faces counted by $x_3$, $p_5$ faces counted by $x_4$ and so on. Thus we are interested in the number
$m_{n-D,\bar{\pt{k}}}$, 
where $\bar{\pt{k}} = (s+p_1,p_2,p_3+p_4,p_5,p_6,p_7,p_8,p_9+p_{10},p_{11}+p_{12})$. Once again this number is easily determined by Lemma~\ref{LeMm}. 

However, when we reconstruct how many maps reduce to this map with $\bar{\pt{k}}$ distinguished faces, then we first have to choose $s$ simple $4$-gons out of the $s+p_1$ selected, into which we reinsert the edges of single koalas. That is, we obtain a factor $\binom{s+p_1}{s}$. Similarly, we choose $p_3$ of the selected faces that are associated to types $3$ and $4$ such that we obtain a factor $\binom{p_3+p_4}{p_3}$ and we do the same for intersection types 9 and 10 and further types 11 and 12, where we obtain the factors $\binom{p_9+p_{10}}{p_9}$ and $\binom{p_{11}+p_{12}}{p_{11}}$ respectively. 

Further, we gain a factor $2^s$ when we reinsert edges of single koalas and according to Lemma~\ref{prop:fact}, each intersection type $i$ in the sum is weighted by $\left(\frac{r_i}{c_{t(i)}}\right)^{p_i}$.

Finally, we should not forget to relabel the $k$ koalas by $1$ to $k$. That is, we gain another factor of $k!$. In the end, we have
\begin{align}\label{eq:justeq}
    \frac{m_{n,k}^{\circ,\times}}{m_n} 
     &= k!\sum_{s+2\sum p_i = k} \binom{s+p_1}{s} \binom{p_3+p_4}{p_3}\binom{p_9+p_{10}}{p_9} \binom{p_{11}+p_{12}}{p_{11}} \nonumber  \\
     &\qquad \qquad \qquad  \qquad \cdot \left( 2^s \prod_{i=1}^{16} \left(\frac{r_i}{c_{t(i)}}\right)^{p_i}\right)\frac{m_{n-D,\bar{\pt{k}}}}{m_{n}}.
\end{align}
Once again, we apply Lemma \ref{LeMm} and deduce that
\begin{align*}
    \E\left[(\pt{Y}_{n-D})_{\bar{\pt{k}}}\right] 
    &=(s+p_1)!p_2!(p_3+p_4)!p_5!p_6!p_7!p_8!(p_9+p_{10})!(p_{11}+p_{12})\cdot \frac{m_{n-D,\bar{\pt{k}}}}{m_{n-D}}
    \\
    &= \left(n-D\right)^{s+\sum p_i}\cdot f_{x_1}(\pt{1})^s\left(\prod_{i=1}^{16} f_{x_i}(\pt{1})^{p_i}\right)e^{\frac{1}{2\left(n-D\right)}\left<\bar{\pt{k}},\Sigma \bar{\pt{k}}\right>}\left(1+O\left(\frac{1}{k}\right)\right)
\end{align*}
with $(\Sigma)_{i,j} = \frac{f_{x_{i}x_{j}}(\pt{1})}{f_{x_{i}}(\pt{1})f_{x_{j}}(\pt{1})}$. 

Since we can rewrite the binomial coefficients divided the factorials by $\frac{1}{(s+p_1)!}\binom{s+p_1}{s} = \frac{1}{s!p_1!}$ for example, we can simplify (\ref{eq:justeq}) to
\begin{align*}
    \frac{m_{n,k}^{\circ,\times}}{m_n} 
    &= k!\sum_{s+2\sum p_i = k} \left(\frac{2^s}{s!}
    \prod_{i=1}^{16}
    \frac{r_i^{p_i}}{c_{t(i)}^{p_i}p_i!} \right) 
    \frac{m_{n-D}}{m_{n}} \;
    \E \left[(\pt{Y}_n)_{\bar{\pt{k}}}\right].
\end{align*}

It is left to entangle this expression and extract the main asymptotic terms. We apply Lemma~\ref{le:frac_map} and fix $s=k-2P$, where $P=\sum p_i$
such that we obtain
\begin{align*}    
    \frac{m_{n,k}^{\circ,\times}}{m_n} &= k!\sum_{P = 0}^{\lfloor k/2 \rfloor} \sum_{\sum p_i = P}
    \left(\frac{2^{k-2P}}{(k-2P)!}\prod_{i=1}^{16}\frac{r_i^{p_i}}{c_{t(i)}^{p_i}p_i!}\right) \frac{m_{n-D}}{m_{n}} \;\E \left[(\pt{Y}_n)_{\bar{\pt{k}}}\right]
    \\
    &= k!\sum_{P = 0}^{\lfloor k/2 \rfloor} \sum_{\sum p_i = P}
    \frac{2^{k-2P}}{(k-2P)!}\left(n-D\right)^{k-P}f_{x_1}(\pt{1})^{k-2P}\\
    &\hspace{30mm} \cdot 
    \left(\prod_{i=1}^{16} \frac{1}{p_i!} \left(\frac{r_if_{x_i}(\pt{1})}{c_{t(i)}}\right) ^{p_i}\right)e^ {\frac{1}{2\left(n-D\right)}\left<\bar{\pt{k}},\Sigma \bar{\pt{k}}\right>}\left(1+O\left(\frac{1}{k}\right)\right)\\
    &\hspace{30mm} \cdot \frac{1}{16^{D}}\left(1+\frac{5D}{2n}+O\left(\frac{1}{n}\right)\right).
\end{align*}
Next, we factor out anything that is independent of the $p_i$'s and distribute anything with exponent $P$ in the product over the $p_i$'s.
\begin{align*} 
    \frac{m_{n,k}^{\circ,\times}}{m_n} &= \left(\frac{2}{12^2}nf_{x_1} (\pt{1})\right)^k\sum_{P = 0}^{\lfloor k/2 \rfloor} \sum_{\sum p_i = P}
   k(k-1)(k-2)\cdots (k-2P+1)\\
    &\qquad \cdot 
    \left(1-\frac{D}{n}\right)^{k}\prod_{i=1}^{16} \frac{1}{p_i!}\left(\frac{r_if_{x_{t(i)}}(\pt{1})}{4c_{t(i)}(n-D)f_{x_1}(\pt{1})^2}\right)^{p_i} \\
    &\qquad \cdot e^{\frac{\left<\bar{\pt{k}}, \Sigma \bar{\pt{k}}\right>}{2(n-D)}} 12^{p_9+p_{11}}\left(1+O\left(\frac{1}{k}\right)\right)
\end{align*}
In the next step, we want to expand the exponential function and rewrite
\begin{align*}
    \left<\bar{\pt{k}}, \Sigma \bar{\pt{k}}\right> &= \left<\pt{k}, \bar{\Sigma} \pt{k}\right>
\end{align*}
where $\bar{\Sigma}$ is an $17\times17$ matrix with entries $(\bar{\Sigma})_{i+1,j+1} = \frac{f_{x_{t(i)}x_{t(j)}}(\pt{1})}{f_{x_{t(i)}}(\pt{1})f_{x_{t(j)}}(\pt{1})}$ for $0\leq i,j \leq 16$ and $t(0)=1$. Then we split the expression such that
\begin{align*}
    \left<\pt{k}, \bar{\Sigma} \pt{k}\right> &= (k^2-4kP+4P^2)\frac{f_{x_1x_1}(\pt{1})}{f_{x_1}(\pt{1})^2}\\
    &\quad +(k-2P)\sum_{i=1}^{16}\frac{f_{x_{1}x_{t(i)}}(\pt{1})}{f_{x_{1}}(\pt{1})f_{x_{t(i)}}(\pt{1})}p_i +\sum_{i,j=1}^{16} \frac{f_{x_{t(i)}x_{t(j)}}(\pt{1})}{f_{x_{t(i)}}(\pt{1})f_{x_{t(j)}}(\pt{1})}\,p_ip_j\\
    &:= k^2\frac{f_{x_1x_1}(\pt{1})}{f_{x_1}(\pt{1})^2}+\sum_{i=1}^{16} q_ip_i 
\end{align*}
where $q_i= 4(P-k)\frac{f_{x_1x_1}(\pt{1})}{f_{x_1}(\pt{1})^2}+(k-2P)\frac{f_{x_{1}x_{t(i)}}(\pt{1})}{f_{x_{1}}(\pt{1})f_{x_{t(i)}}(\pt{1})}+\sum_{j=1}^{16} \frac{f_{x_{t(i)}x_{t(j)}}(\pt{1})}{f_{x_{t(i)}}(\pt{1})f_{x_{t(j)}}(\pt{1})}p_j$. Since $q_i = O(k)$, the exponential function in the modified expression below is asymptotically equal to 
\begin{align*}
    \exp\left(\frac{\left<\bar{\pt{k}}, \Sigma \bar{\pt{k}}\right>}{2(n-D)}\right) &= \exp\left(\frac{1}{2n}\left(k^2\frac{f_{x_1x_1}(\pt{1})}{f_{x_1}(\pt{1})^2}+\sum_{i=1}^{16} q_ip_i \right)\right)\left(1+O\left(\frac{1}{k}\right) \right).
\end{align*}
The above therefore reduces to
\begin{align*}
    \frac{m_{n,k}^{\circ,\times}}{m_n} &= \left(\frac{2}{12^2}nf_{x_1} (\pt{1})\right)^k e^{\frac{k^2f_{x_1x_1}(\pt{1})}{2nf_{x_1}(\pt{1})^2}}\sum_{P = 0}^{\lfloor k/2 \rfloor} \sum_{\sum p_i = P}
   \frac{k(k-1)\cdots (k-2P+1)}{k^{2P}}\\
    &\qquad \cdot 
    \left(1-\frac{D}{n}\right)^{k}\prod_{i=1}^{16} \frac{1}{p_i!}\left(\frac{k^2}{n}\frac{r_if_{x_{t(i)}}(\pt{1})}{4c_{t(i)}f_{x_1}(\pt{1})^2} e^{\frac{q_i}{2n}}\right)^{p_i}\\
    &\qquad \cdot 12^{p_9+p_{11}}\left(1+O\left(\frac{1}{k}\right)\right).
\end{align*}
Since  $D=2k-p_9-p_{11}$ we can write
\[\left(1-\frac{D}{n}\right)^k = \exp\left(\frac{-2k^2}{n}\right)\left(1+O\left(\frac{(P+1)k}{n}\right)\right)\left(1+O\left(\frac{k^3}{n^2}\right)\right) = e^{-\frac{2k^2}{n}}\left(1+O\left(\frac{P+1}{k}\right)\right) \]
such that eventually we end up with
\begin{align*}
    \frac{m_{n,k}^{\circ,\times}}{m_n} &= \left(\frac{2}{12^2}nf_{x_1} (\pt{1})\right)^k e^{\frac{k^2}{n}\frac{\left(f_{x_1x_1}(\pt{1})-4f_{x_1}(\pt{1})^2\right)}{2f_{x_1}(\pt{1})^2}}\\
    &\qquad \cdot \sum_{P = 0}^{\lfloor k/2 \rfloor} \sum_{\sum p_i = P}
   \frac{k(k-1)\cdots (k-2P+1)}{k^{2P}}\left(1+O\left(\frac{P+1}{k}\right)\right)\\
    &\qquad \qquad \qquad \qquad \cdot 
    \prod_{i=1}^{16} \frac{1}{p_i!}\left(\frac{k^2}{n}\frac{12^{4-d_i}r_if_{x_{t(i)}}(\pt{1})}{4c_{t(i)}f_{x_1}(\pt{1})^2} e^{\frac{q_i}{2n}}\right)^{p_i}.
\end{align*}
Now given the factorials of $p_i$ in the denominators in the product, asymptotically we only need to consider the terms where $P$ is small, say $P\le k^{1/3}$ such that we can use the estimate $\frac{k(k-1)\cdots (k-2P+1)}{k^{2P}} = 1+ O\left(\frac{P^2}{k}\right)$. For $P > k^{1/3}$ we can use
the trivial upper bound $\frac{k(k-1)\cdots (k-2P+1)}{k^{2P}} \le 1$ and observe directly
that the corresponding terms are negligible. We also use the bound
and $e^{\frac{q_i}{2n}} = 1+O\left(\frac{k}{n}\right)$ and derive
\begin{align*}
    \frac{m_{n,k}^{\circ,\times}}{m_n}  &\sim \left(\frac{2}{12^2}nf_{x_1} (\pt{1})\right)^k e^{\frac{k^2}{n}\frac{\left(f_{x_1x_1}(\pt{1})-4f_{x_1}(\pt{1})^2\right)}{2f_{x_1}(\pt{1})^2}}\\
    &\quad \cdot \sum_{P \le k^{1/3}} \sum_{\sum p_i = P}
    \prod_{i=1}^{16} \left(\frac{1}{p_i!}\left(\frac{k^2}{n}\frac{12^{4-d_i}r_if_{x_{t(i)}}(\pt{1})}{4c_{t(i)}f_{x_1}(\pt{1})^2} \left( 1+ O\left(\frac{k}{n}\right)\right)\right)^{p_i}\right)\left( 1+ O\left(\frac{P^2}{k}\right)\right)\\
    &= \left(\frac{2}{12^2}nf_{x_1} (\pt{1})\right)^k e^{\frac{k^2}{2n}\left(\frac{f_{x_1x_1}(\pt{1})-4f_{x_1}(\pt{1})^2}{nf_{x_1}(\pt{1})^2}+
    \sum_{i=1}^{16} \left(\frac{12^{4-d_i}r_if_{x_{t(i)}}(\pt{1})}{2c_{t(i)}f_{x_1}(\pt{1})^2}\right)\right)}\left( 1+ O\left(\frac{1}{k^{1/3}}\right)\right)
\end{align*}
Finally, recall that $\mu_n \sim \frac{2}{12^2}nf_{x_1} (\pt{1})$ such that we can derive
\begin{align*}
    \frac{m_{n,k}^{\circ,\times}}{m_n} & \sim \mu_n^k\exp\left(\frac{k^2}{2n}\left(\frac{f_{x_1x_1}(\pt{1})-4f_{x_1}(\pt{1})^2}{f_{x_1}(\pt{1})^2}+
    \sum_{i=1}^{16} \left(\frac{12^{4-d_i}r_if_{x_{t(i)}}(\pt{1})}{2c_{t(i)}f_{x_1}(\pt{1})^2}\right)\right)\right).
 \end{align*}

This looks very promising with respect to Theorem \ref{thm:gao_wormald} and it is straightforward to compute that
\begin{align*}
    \frac{k^2}{2}\frac{\mathbb{V}(X_n) -  \E(X_n)}{ \E(X_n)^2}
    &= \frac{k^2}{2}\frac{\left(\frac{4}{12^4}f_{x_{1}x_{1}}(\pt{1})-\frac{16}{12^4}f_{x_{1}}(\pt{1})^2  + \sum_{i=1}^{16}\frac{2r_i}{12^{d_i}c_{t(i)}}f_{x_{t(i)}}(\pt{1})  \right)n + O(1)}{\frac{4}{12^4} f_{x_{1}}(\pt{1})^2n^2\Big(1 + o(1)\Big)}\\
    &\sim \frac{k^2}{2n}\frac{2f_{x_{1}x_{1}}(\pt{1})-8f_{x_{1}}(\pt{1})^2  + \sum_{i=1}^{16} 12^{4-d_i}\frac{r_i}{c_{t(i)}}f_{x_{t(i)}}(\pt{1})  }{2f_{x_{1}}(\pt{1})^2}
\end{align*}
just as desired.
Concerning Lemma~\ref{lem:S1}, we note that $\frac{m_{n,k}^{\circ,\times}}{m_n}$ is indeed the main contribution to the asymptotics of the $k$-th moment because $(\mu_n/2)^k = o(\mu_n^k)$ which completes the proof of Theorem \ref{thm:central_koala}.

\section{Main theorem}\label{sec:main}

In this section, we extend the procedure which we used in the previous section to pattern occurrences of arbitrary maps with simple boundary and prove the following theorem.

\begin{theo}\label{thm:central}
    Let \map{p} be a planar map with simple boundary and let $X_n$ be the number of occurrences of \map{p} as a pattern in a random rooted planar map with $n$ edges. Then
    \[\frac{X_n - \E[X_n]}{\sqrt{\V[X_n]}} \;\longrightarrow\; \mathcal{N}(0,1)\]
    where $\E[X_n] \sim \mu n$, $ \V[X_n] \sim \sigma^2n$ and $\mu, \sigma \in \R$.
\end{theo}

As in the previous section, first we prove that $\E [X_n] $ and $\V [X_n]$ are both linear in $n$ and conclude that we need to compute the factorial moments of order $k = O(\sqrt{n})$ to satisfy the conditions of Theorem~\ref{thm:gao_wormald}.
The core of this computation is the estimation of the number of maps with $k$ selected pattern occurrences that are labelled from $1$ to $k$ and where each selected pattern occurrence intersects at most on other pattern occurrence, since by Lemma \ref{lem:S1}
\begin{equation*}
    \frac{m_{n,k}^{\circ,\times} }{m_{n}}\quad \leq \quad \E\left[(X_n)_k\right] = \frac{m^{\circ}_{n,k}}{m_n} \quad \leq\quad   \frac{       m_{n,k}^{\circ,\times}}{m_{n}}+\left(\frac{\mu_n}{2}\right)^k+o\Big(\E\left[(X_n)_k\right]\Big)
    \tag{\ref{eq:kthmom2}}
\end{equation*}
where $m_{n,k}^\circ$ is the number of maps with $n$ edges and $k$ selected and labelled pattern occurrences (among arbitrary many occurrences) and $m_{n,k}^{\circ,\times}$ is the subset of all such maps where each selected pattern occurrence intersects at most one other selected occurrence.

Next we transform the problem of estimating the asymptotic growth of $m_{n,k}^{\circ,\times}$ to the asymptotic estimation of the factorial moments of face counts. The outline is exactly the same as for koalas in Section~\ref{sec:koalas}. 

\begin{itemize}
    \item We list the all intersection types $i\in [I]$ which may appear and define a maximal set of interior edges of the intersecting patterns which shall be deleted such that the map does not disconnect. That is, we fix the associated maps  respectively and distinguish the face which contained the interior edges. See Figure~\ref{fig:intersects} in the case of koalas. 
    \item For each intersection type $i \in [I]$, $d_i$ denotes the number of edges which are deleted in the associated map. 
    \item Let $\ell_0$ be the root face valency of $\map{p}$. Then the associated map of $\map{p}$ is a rooted simple $\ell_0$-gon. Further, let $r_0$ be the number of rotations of $\map{p}$ and $d_0$ the number of interior edges of $\map{p}$.
    \item The associated maps of two distinct intersection types may be the same. So, let $J$ be the number of distinct associated maps. We introduce a function $t:\{0,1,2,\dots,I\}\mapsto \{1,2,\dots,J\}$, such that the associated maps of intersection type $i$ are given by a rotation of $\map{f}_{t(i)}$ and the associated map of a rotation of $\map{p}$ is given by $\map{f}_{t(0)}$.
    \item Again, for each intersection type $i\in [I]$, $r_i$ denotes the number of rotations of the intersection type and $c_{t(i)}$ denotes the number of rotations of the associated map.
    \item Further, we define the generating function of rooted planar maps $M(z,u,\pt{x})$ where $z$ counts the number of edges, $u$ the root face valency of the map and the vector $\pt{x} = (x_1,x_2,\dots,x_J)$ consists of additional variables, where $x_i$ counts the number of faces which have the same boundary as the distinguished face in the associated maps.
    The boundary shape of the faces counted by $x_i$ is denoted by $(\ell_i, e_i, s_{i1},s_{i2},s_{i3},\dots)$.
     
    \item The boundary shape $(\ell_i, e_i, s_{i1},s_{i2},s_{i3},\dots)$ of the distinguished face  associated to intersection type $i$ translates to the term
    \[
        (x_{i}-1)c_{t(i)}z^{e_i+1}u^{2-\ell_i}P_{\ell_i-1}(z,u,\pt{x})\prod_{j\geq 1}S_j^{s_{ij}}(z,\pt{x})
    \]
    in the functional equation for the generating function of rooted planar maps $M(z,u,\pt{x})$ with additional variables $\pt{x}$ which count the faces associated to the intersection types.
\end{itemize}

Then the generating functions for the maps with additional face counting variable can be described in terms of the equation
\begin{align}\label{eq:gen_pats}
    M(z,u,\pt{x}) &= 1 + zu^2M(z,u,\pt{x})+zu\frac{uM(z,u,\pt{x})-M(z,1,\pt{x})}{u-1} \\
    &\quad + (x_{t(0)}-1)\frac{z}{u^{\ell_0-2}}P_{\ell_0-1}(z,u,\pt{x}) \nonumber\\
    &\quad + \sum_{i=1}^I(x_{t(i)}-1)c_{i}z^{e_{t(i)}+1} u^{2-\ell_{t(i)}}P_{\ell_{t(i)}-1}(z,u,\pt{x})\prod_{j\geq 1} S_j^{s_{t(i)j}}(z,\pt{x}) \nonumber
\end{align}
where $P_{\ell}(z,u,\pt{x})$ and $S_{\ell}(z,\pt{x})$ are the generating functions of $P_\ell$- and $S_\ell$-maps.
Given these parameters, we can refine the computations for small moments, like expectation and variance.
\begin{lemma}\label{th:exp_var}
    Let \map{p} be a map with simple boundary and $X_n$ be the number of occurrences of \map{p} as a pattern in a random rooted planar map with $n$ edges. Then
    \begin{align*}
        \E[X_n] &\sim \frac{r_0}{12^{d_0}} f_{x_{t(0)}}(\pt{1}) n\\
        \mathbb{V}[X_n] &\sim \frac{1}{12^{2d_0}}\Bigg(r_0^2f_{x_{t(0)}x_{t(0)}}(\pt{1})-2d_0r_0^2f_{x_{t(0)}}(\pt{1})^2 +12^{d_0}r_0f_{x_{t(0)}}(\pt{1})+\sum_{i=1}^I \frac{2r_i}{12^{d_i}c_{t(i)}}f_{x_{t(i)}}(\pt{1})\Bigg)n
    \end{align*}
    where $d_i,r_i,t(i),c_{t(i)}$ and $I$ are defined as above and $f_{x_{t(i)}}(\pt{1})$ are computable constants.
\end{lemma}
Again, we refer to Section \ref{sec:proofs} for computations on the constants.

\section{Proofs}\label{sec:proofs}
\subsection{Proof of Lemma \ref{LeMm}}

We focus on the one- and two-dimensional cases $m = 1$ and $m=2$. The general case can be proven analogously. We start with $m=1$.

    Since
    \[
        \E[ x^{X_n}] = \sum_{k\ge 0} \E[(X_n)_k] \frac{(x-1)^k}{k!}
    \]
    we can represent the factorial moments $\E[(X_n)_k]$ with the help of a
    Cauchy integral:
    \[
        \E[(X_n)_k] = k! \frac 1{2\pi i} \int_\gamma \frac {\E[ x^{X_n}]}{(x-1)^{k+1}}\, dx,
    \]
    where $\gamma$ is a contour encycling $x=1$. In the present context $\gamma$ will be a cycle with
    radius $k/(nf'(1))$. Thus, if we substitute $x = 1 + r e^{i\varphi}$ with $r = k/(nf'(1))$ and by
    using (\ref{eq:power}) we obtain
    \begin{align*}
        \E[(X_n)_k] &= \frac{k!}{2\pi } \int_{-\pi}^\pi e^{n f'(1)r e^{i\varphi} + \frac n2 f''(1)r e^{2i\varphi} 
        + O(nr^3 + r + 1/n) } r^{-k} e^{-ik\varphi}\, d\varphi \\
        & = (nf'(1))^k  e^{\frac{k^2}{2n} \frac{f''(1)}{(f'(1))^2} }  \frac{k!}{2\pi k^{k}e^{-k}}  
        \int_{-\pi}^\pi  e^{k (e^{i\varphi} - 1 - i\varphi) + \frac{k^2}{2n} \frac{f''(1)}{(f'(1))^2} (e^{2i\varphi} -1) 
        + O( \frac{k^3}{n^2} + \frac{k}{n} )} \, d\varphi.
    \end{align*}
    The last integral can be easily asympotically evaluated (uniformly for $1\le k\le C\sqrt n$, where we use
the substition $z = e^{i\varphi}$ on a circle $\gamma'$ around $z=0$ and use the property that 
 $\Re(e^{i\varphi} - 1 - i\varphi)\le -c \varphi^2$ for some comstant $c> 0$ and $|\varphi| \le \pi$ ):
    \begin{align*}
        \frac 1{2\pi} \int_{-\pi}^\pi  & e^{k (e^{i\varphi} - 1 - i\varphi) + \frac{k^2}{2n} \frac{f''(1)}{(f'(1))^2} (e^{2i\varphi} -1) 
        + O( \frac{k^3}{n^2} + \frac{k+1}{n}) } \, d\varphi \\
        &= \frac 1{2\pi} \int_{-\pi}^\pi  e^{k (e^{i\varphi} - 1 - i\varphi)} \, d\varphi 
       + O\left( \int_{-\pi}^{\pi} e^{-c k \varphi^2} \left( |\varphi| \frac{k^2}n + \frac{k^3}{n^2} + \frac{k}{n}  
       \right) d\varphi      \right)  \\
        &= \frac {e^{-k}}{2\pi i} \int_{\gamma'} e^{kz} z^{-k-1}\, dz +
        O\left( \frac kn + \frac{k^{5/2}}{n^2} + \frac{k^{1/2}}n   \right) \\   
        &=  \frac{e^{-k} k^k}{k!} + O\left( \frac{k} n  \right) 
    \end{align*}
    So that we finally obtain 
\[
     \E[ x^{X_n}] =  (nf'(1))^k  e^{\frac{k^2}{2n} \frac{f''(1)}{(f'(1))^2} } \left( 1 + O\left( \frac{k^{3/2}}n  \right) \right)
\]   
which completes the case $m=1$.

Next suppose that $m=2$, where we use the notation $(X_n,Y_n)$ for the random vector and $k_1 = k$ and $k_1 = \ell$.
Here we have
    \[
        \E[(X_n)_{k}(Y_n)_\ell] = k! \ell!  \frac 1{(2\pi i)^2} \iint\limits_{\gamma_1\times \gamma_2} 
        \frac {\E[ x_1^{X_n} x_2^{Y_n}  ]}{(x_1-1)^{k+1}(x_2-1)^{\ell+1}}\, dx_1dx_2,
    \]
where $\gamma_i$, $i=1,2$ is a circle around $x_i=1$ with radius 
$k/(nf_{x_1}(1,1))$ for $i=1$ and radius $\ell/(nf_{x_2}(1,1))$ for $i=2$.
If we substitute $x_1 = 1 + r_1 e^{i\varphi_1}$ with $r_1 = k/(nf_{x_1}(1,1))$ and
$x_2 = 1 + r_2 e^{i\varphi_2}$ with $r_2 = \ell/(nf_{x_2}(1,1))$ we, thus obtain (by using (\ref{eq:power}))
    \begin{align*}
        \E[(X_n)_k(Y_n)_\ell]
        & = (nf_{x_1}(1,1))^k (nf_{x_2}(1,1))^\ell  
        e^{ \frac{k^2}{2n} \frac{f_{x_1x_1}(1,1)}{(f_{x_1}(1,1))^2} + \frac{k\ell}{n} \frac{f_{x_1x_2}(1,1)}
        {f_{x_1}(1,1) f_{x_2}(1,1)} + 
        \frac{\ell^2}{2n} \frac{f_{x_2x_2}(1,1)}{(f_{x_2}(1,1))^2}    }   \frac{k! \ell !}{k^{k} \ell^\ell e^{-k-\ell}}  \\
        &\quad \cdot  \frac{1}{(2\pi)^2} 
        \iint_{[-\pi,\pi]^2}  e^{k (e^{i\varphi_1} - 1 - i\varphi_1) + \ell(e^{i\varphi_2} - 1 - i\varphi_2) + A  
        + O( \frac{k^3+\ell^3}{n^2} + \frac{k+\ell}{n} )} \, d\varphi_1 d\varphi_2, 
    \end{align*}
where $A$ abbreviates
\begin{align*}
    A &= \frac{k^2}{2n} \frac{f_{x_1x_1}(1,1)}{(f_{x_1}(1,1))^2} (e^{2i\varphi_1} -1) 
    + \frac{k\ell}{n} \frac{f_{x_1x_2}(1,1)}{f_{x_1}(1,1)f_{x_2}(1,1)} (e^{i\varphi_1+\varphi_2} -1) 
    + \frac{\ell^2}{2n} \frac{f_{x_2x_2}(1,1)}{(f_{x_2}(1,1))^2} (e^{2i\varphi_2} -1) \\
    &= O\left( \frac{k^2}n |\varphi_1| + \frac{k\ell}n (|\varphi_1|+|\varphi_2|) + \frac{\ell^2}n |\varphi_2|   \right).    
\end{align*}
As above the integral can be directly evaluated:
\[
    \frac{1}{(2\pi)^2} 
            \iint\limits_{[-\pi,\pi]^2}  e^{k (e^{i\varphi_1} - 1 - i\varphi_1) + \ell(e^{i\varphi_2} - 1 - i\varphi_2)} 
    \, d\varphi_1 d\varphi_2   = \frac{e^{-k-\ell} k^k \ell^\ell}{k! \ell!}.
\]
Furthermore we have
\begin{align*}
    \frac{k! \ell !}{k^{k} \ell^\ell e^{-k-\ell}} & \iint\limits_{[-\pi,\pi]^2}  
    e^{-c(k\varphi_1^2 + \ell \varphi_2^2)}
     O \left(  \frac{k^2}n |\varphi_1| + \frac{k\ell}n (|\varphi_1|+|\varphi_2|) + \frac{\ell^2}n |\varphi_2| 
    +  \frac{k^3+\ell^3}{n^2} + \frac{k+\ell}{n}  \right)  \\
    & = O\left( \frac{k^{3/2}}n + \frac{\ell k^{1/2}}n + \frac{\ell^{1/2} k}n + \frac{\ell^{3/2}}n 
    + \frac{k^3+\ell^3}{n^2} + \frac{k+\ell}{n}  \right) \\
    & = O\left( \frac{k^{3/2}}n + \frac{\ell^{3/2}}n  \right) 
\end{align*}
uniformly for $1\le k, \ell \le C \sqrt n$. This directly leads to 
\begin{align*}
     \E[(X_n)_k(Y_n)_\ell]
             = (nf_{x_1}(1,1))^k &  (nf_{x_2}(1,1))^\ell  
            e^{ \frac{k^2}{2n} \frac{f_{x_1x_1}(1,1)}{(f_{x_1}(1,1))^2} + \frac{k\ell}{n} \frac{f_{x_1x_2}(1,1)}
      {(f_{x_1}(1,1))(f_{x_2}(1,1))} + 
    \frac{\ell^2}{2n} \frac{f_{x_2x_2}(1,1)}{(f_{x_2}(1,1))^2}    } \\
    &\cdot \left( 1 + O\left( \frac{k^{3/2}}n + \frac{\ell^{3/2}}n  \right) \right)
\end{align*}
as proposed.     

\subsection{Proof of Lemma \ref{lem:S1}}

First, we distinguish maps according to the number of patterns that appear. Let $S_1$ be the set of maps with more than $\mu_n/2$ pattern occurrences and $S_2$ the set of maps with less pattern occurrences. Then the contribution of maps in $S_2$ with $k$ labelled patterns to $m_{n,k}$ is at most 
\[
    (\mu_n/2)^k|S_2|\leq (\mu_n/2)^km_n
\] 
since there are at most $(\mu_n/2)_k$ ways to mark $k$ patterns in a map of the set $S_2$.

Next, we have a closer look at the contribution of the maps in $S_1$. 

In this case an important observation is that a pattern occurrence of \map{p} in a map \map{m} can only intersect a constant number $c:= vf^2$ of other pattern occurrences, where $f$ is the number of interior faces of \map{p} and $v$ is the maximal face valency of the interior faces. This follows because a face $F$ in \map{m} can intersect with at most $f$ different faces of an occurrence of \map{p} and the intersection can be rotated in at most $v$ ways. Applying this estimate to all faces of a fixed pattern occurrence, we obtain the desired bound.

So, let \map{m} in $S_1$ with $p$ occurences of the pattern \map{p} and let $L(\map{m})$ be the set of maps consisting of the map \map{m} with $k$ labelled pattern occurrences. Since there are $(p)_k$ ways to label $k$ pattern occurrences the size of $L(\map{m})$ is exactly $(p)_k$.

However, the fraction of maps in $L(\map{m})$ with a labelled pattern occurrence that intersects with at least two other labelled occurrences is bounded by $O\left(\frac{1}{\sqrt{n}}\right)$.
To that end, let us estimate from above the number of ways to label $k$ patterns where at least one labelled occurrence intersects at least two other occurrences. 

Naturally there are at most $p$ choices for the pattern occurrence which intersects two others and $k$ choices for its label. By the observation above, there are at most $c^2$ choices for the two pattern occurrence that the labelled occurrence intersects and $(k-1)$ and $(k-2)$ choices for their labels respectively. The rest of the labelled pattern occurrences can be chosen freely and we obtain that there are in total at most
\[
    (pk)\cdot \left(c(k-1)\right)\cdot \left(c(k-2)\right)\cdot (p-3)_{(k-3)}
\]
ways to label $k$ pattern occurrences such that there exists at least one labelled occurrence that intersects two other labelled occurrences. Comparing this number to the size of $L(\map{m})$ yields
\[
    \frac{c^2k(k-1)(k-2)p(p-3)(p-4)\cdots (p-k+1)}{p(p-1)\cdots (p-k+1)} \leq \frac{c^2k^3}{(p-2)^2} \leq \frac{4c^2k^3}{(\mu_n-4)^2} \leq  C\frac{1}{\sqrt{n}}
\]
for a suitably large constant $C$ and all $n$ large enough, since $k = \Theta(\sqrt{n})$ and $\mu_n = \Theta(n)$ by Lemma \ref{th:exp_var}.

By using the rough estimates $|S_1| < m_n$ and $\sum_{\ell \geq \mu_n/2} (\ell)_k\, m_{n,\ell} < m_{n,k}^\circ$, the considerations above sum up to
\[
     m_{n,k}^\circ \leq m_{m,k}^{\circ,\times} + \left(\frac{\mu_n}{2}\right)^k m_{n} + O\left(\frac{m_{n,k}^\circ}{\sqrt{n}}\right) 
\]

\subsection{Functional equations and their solutions}\label{sec:quasi_power} 
In this section, we have a closer look at the functional equation systems (\ref{eq:gen_pats}) associated to the generating function of planar maps 
\[
    M(z,u,\pt{x}) = \sum_{n\geq 0} M_n(z,\pt{x})u^n
\]
where $z$ counts the number of edges, $u$ the root face valency and $\pt{x}$ is a vector of face counting variables. We consider the full system consisting of equation (\ref{eq:gen_pats}) and the formulas of Lemma \ref{lem:Yu} and Lemma \ref{lem:simple}
\begin{align*}
    M(z,u,\pt{x}) &= 1 + zu^2M(z,u,\pt{x})+zu\frac{uM(z,u,\pt{x})-M(z,1,\pt{x})}{u-1} \\
    &\qquad + (x_{t(0)}-1)\frac{r_0z}{u^{\ell_0-2}}P_{\ell_0-1}(z,u,\pt{x}) \nonumber\\
    &\qquad + \sum_{i=1}^I(x_{t(i)}-1)c_{t(i)}z^{e_{t(i)}+1} u^{2-\ell_{t(i)}}P_{\ell_{t(i)}-1}(z,u,\pt{x})\prod_{j\geq 1} S_j^{s_{t(i)j}}(z,\pt{x}) \nonumber\\
    P_0(z,u,\pt{x}) &= M(z,u,\pt{x}) \nonumber\\
    P_{\ell}(z,u,\pt{x}) &= M(z,u,\pt{x})-\sum_{k=0}^{\ell-1}M_k(z,\pt{x})u^k \\
    &\qquad - \sum_{k=0}^{\ell-1} P_{k}(z,u,\pt{x})u^{\ell-k}[u^{\ell-k}] M(z,u,\pt{x})^{k+1}, \quad \ell \geq 1 \nonumber\\ 
    S_1(z,\pt{x}) &= M_1(z,\pt{x})\\
    S_2(z,\pt{x}) &= M_2(z,\pt{x})-M_1(z,u)^2-z\\
    S_\ell(z,\pt{x}) &= M_\ell(z,\pt{x}) - \sum_{k=1}^{\ell-1}S_k(z,\pt{x})[u^{\ell-k}]M(z,u,\pt{x})^k - z[u^{\ell-2}]M(z,u,\pt{x})^2, \quad \ell \geq 3.
\end{align*}
Note that $[u^{\ell-k}] M(z,u,\pt{x})^{k+1}$ is a polynomial in $M_0(z,\pt{x}), M_1(z,\pt{x}), \dots M_{\ell-k}(z,\pt{x})$
and consequently, the functions $P_\ell(z,u,\pt{x})$ and $S_\ell(z,\pt{x})$ as well. Further, we only need to compute $P_\ell(z,u,\pt{x})$ and $S_\ell(z,\pt{x})$ up to $\ell < 2\ell_0-2$ since the boundary of two intersecting patterns is of length at most $2\ell_0-2$. 

The equations for these functions are also iterative such that we can directly express them in terms of the polynomials
\begin{align*}
    u^{1-\ell}P_\ell(z,u,\pt{x}) &= Q_\ell(z,u,M(z,u,\pt{x}),M_0(z,\pt{x}),\dots, M_\ell(z,\pt{x}))&\\
    S_\ell(z,\pt{x}) &= T_\ell(z,M_0(z,\pt{x}),\dots, M_\ell(z,\pt{x})).
\end{align*}

 In turn, the generating functions of $M_i(z,\pt{x})$ are generally only given implicitly. However, we refer to~\cite{DrPa} where it is proven that they are given by $M_i(z,\pt{x}) = G_{i}(z,\pt{x},M(z,1,\pt{x}))$, where $G_{i}(z,\pt{x},y)$ is an analytic function for $|z|<\frac{1}{10}$, $|y|<2$, $|\pt{x}-\pt{1}|<2^{1-i}$.  Plugging these functions into the resulting DDE, we can apply Theorem~4 in~\cite{DrmotaNoyYu} and conclude that $M(z,1,\pt{x})$ has a $3/2$-singularity at its radius of convergence $r(\pt{x})$. That is locally
\[
    M(z,1,\pt{x}) = H_1(z,\pt{x})+H_2(z,\pt{x})\left(1-\frac{z}{r(\pt{x})}\right)^{\frac{3}{2}},
\]
where $H_1(z,\pt{x})$ and $H_2(z,\pt{x})$ are analytic functions around $(z,\pt{x}) = (r(\pt{1}), \pt{1})$. In particular, equation~(\ref{eq:power}) holds for the probability generating function of the face counting random variables such that we can apply Lemma~\ref{LeMm}.

\subsubsection{A short note on the computation of the constants}~\\
At this point, we comment on the computation of the constants $f_{x_i}(\pt{1})$ appearing in Theorems \ref{thm:central}, Lemma \ref{th:exp_var} and Proposition \ref{prop:nonself}. The function $f(\pt{x})$ is defined as 
\[
    f(\pt{x})=\log(\rho(\pt{1}))-\log(\rho(\pt{x}))
\]
and we want to know the value of its partial derivatives at $\pt{x}=\pt{1}$. Since
\[
    f_{x_i}(\pt{x}) = -\frac{\rho_{x_i}(\pt{x})}{\rho(\pt{x})},\quad f_{x_ix_j}(\pt{x}) = -\frac{\rho_{x_ix_j}(\pt{x})}{\rho(\pt{x})}+\frac{\rho_{x_i}(\pt{x})\rho_{x_j}(\pt{x})}{\rho(\pt{x})^2}.
\]
our goal is to determine the partial derivatives of $\rho(\pt{x})$ at $\pt{1}$. 

First, we apply the method of Bousquet-Mélou and Jehanne~\cite{BMJ} to the equation
\begin{align*}
    M(z,u,\pt{x}) &= 1 + zu^2M(z,u,\pt{x})+zu\frac{uM(z,u,\pt{x})-M(z,1,\pt{x})}{u-1} \\
    &\quad + \sum_{i=0}^I(x_{t(i)}-1)c_{t(i)}z^{e_i}\prod_{j\in J_i} zT_j^{j_i}(...)Q_{b(i)-1}(...).
\end{align*}
where we have omitted the arguments $(z,u,\pt{x},M(z,u,\pt{x}),m_0(z,\pt{x}),\dots m_{2b(0)}(z,\pt{x}))$ to shorten notation. The method consists of expanding the single equation by two further equations: one being the partial derivative with respect to $M(z,u,\pt{x})$, the other with respect to $u$. Drmota, Noy and Yu~\cite{DrmotaNoyYu} further found that substituting $v:= u-1$ yields a positive strongly connected system which in our case is extended by the functions $m_\ell(z,\pt{x}), \ell = 0,1,\dots,2b(0)$. That is,
\begin{align}\label{eq:sys_full_wild}
    g &= 1 + z(v+1)^2g^2+z(v+1)\left(g + w\right)\\
    &\quad + \sum_{i=0}^I(x_{t(i)}-1)c_{t(i)}z^{e_i}\prod_{j\in J_i} zQ_{\ell(i)-1}(z,v+1,g,m_0,\dots, m_\ell)T_j^{j_i}(z,m_0,\dots, m_\ell) \nonumber\\
    v &= 2zv(v+1)^2g +z(v+1)^2\nonumber\\
    &\quad + \sum_{i=0}^I(x_{t(i)}-1)c_{t(i)}z^{e_i}\prod_{j\in J_i} zuQ_{\ell(i)-1, \hat{M}}(z,v+1,g,m_0,\dots, m_\ell)T_j^{j_i}(z,m_0,\dots, m_\ell)\nonumber\\
    w &= 2z(v+1)g^2+2z(v+1)^2gw+z\left(g + w\right)+z(v+1) w\nonumber\\
    &\quad + \sum_{i=0}^I(x_{t(i)}-1)c_{t(i)}z^{e_i}\prod_{j\in J_i} zQ_{\ell(i)-1,u}(z,v+1,g,m_0,\dots, m_\ell)T_j^{j_i}(z,m_0,\dots, m_\ell)\nonumber\\
    &\quad + \sum_{i=0}^I(x_{t(i)}-1)c(i)z^{e(i)}\prod_{j\in J_i} zwQ_{b(i)-1,\hat{M}}(z,v+1,g,m_0,\dots, m_\ell)T_j^{j_i}(z,m_0,\dots, m_\ell)\nonumber
\end{align}
with the additional equations
\begin{align*}
    m_0 &= 1  \nonumber \\
    m_1 &= z(g-wv) \nonumber \\
    m_\ell &= z\sum_{i=0}^{\ell-2}m_i,m_{i-2} + z\left(g-wv-\sum_{i=0}^{\ell-1}m_i\right) \nonumber \\
    &\quad +\sum_{i=0}^I(x_{t(i)}-1)c_{t(i)}z^{e_i}\prod_{j\in J_i} z G_{\ell,i}(z,g-uw)T_j^{j_i}(z,m_0,\dots, m_\ell), \quad 2\leq \ell \leq 2b(i).
     \nonumber
\end{align*}
This system has unique solutions $g(z,\pt{x}),w(z,\pt{x}),v(z,\pt{x}),m_0(z,\pt{x}),m_1(z,\pt{x}),\dots, m_\ell(z,\pt{x})$ with $g(z,\pt{x}) = \hat{M}(z,v(z,\pt{x}),\pt{x})$ and $w(z,\pt{x}) = \Delta \hat{M}(z,u(z,\pt{x}),\pt{x})$. Further, $\rho(\pt{x})$ is defined as the function along which the determinant of the Jacobian of the system with respect to $g,v,w,m_0,m_1,\dots, m_{2b(0)}$ is vanishes. So, we add another equation to the system,   
\[\det J = 0\]
where $J$ denotes the Jacobian of system~(\ref{eq:sys_full_wild}).
Denote the Jacobian of this extended system by $F(\rho,v,g,w,m_0,\dots, m_\ell,\pt{x})$. The matrix $F(\rho,u,g,w,m_0,\dots m_\ell,\pt{x})$ is invertible for $\pt{x}$ close to $\pt{1}$ by standard theory (see e.g. \cite{DrmotaTrees}), such that locally
\[
    \partial_{x_i}(\rho,u,g,w,m_0,\dots m_\ell, \pt{x})^T = F^{-1}\left(S_{x_i}(\rho,u,g,w,m_0,\dots m_\ell, \pt{x}), \partial_{x_i}(\det J(\rho,u,g,w,m_0,\dots m_\ell, \pt{x}))\right)^T
\]
by the implicit function theorem.
Since we can solve the system for $\pt{x}=\pt{1}$, we know the values of $\rho(\pt{1}), g(\pt{1}), w(\pt{1}), u(\pt{1}), m_i(\pt{1})$, which we can simply plug into the equation. An analogous computation yields the second partial derivatives.\\

This computation is slightly more complicated if $G_{\ell,i}(z,M(z,1,\pt{x}),\pt{x})$ is not given explicitly. However, we only need to know its value at $z=\rho(\pt{1})=\frac{1}{12}$, $\pt{x}=\pt{1}$ and the values of $\partial_{x_i}G_\ell(z,M(z,1,\pt{x}),\pt{x})$ at this point. So suppose that $b(i) = 3$ and we want to compute
\[
    G_{2,i}(z,M(z,1,\pt{x}),\pt{x}) = [u^{2+3-2}]P_{3-1}(z,u,\pt{x}).
\]
We know by Lemma~\ref{lem:Yu} that
\[
    P_2(z,u,\pt{x}) = M(z,u,\pt{x})\left(1-u^2m_2(z,\pt{x})-2um_1(z,\pt{x})+2u^2m_1(z,\pt{x})\right) - 1 +um_1(z,\pt{x})
\]
and therefore
\[
    G_{2,i}(z,M(z,1,\pt{1}),\pt{1}) = \left(m_3(z)-3m_2(z)m_1(z)+2m_1(z)^2\right).
\]
The functions $m_1(z),m_2(z),m_3(z)$ in turn can easily be computed from the original map equation of Tutte that we covered in Section~\ref{sec:asymp_normal} by expanding the first few equation like we did for the equation with extra parameters. In particular, their value at $z=\frac{1}{12}$ can easily be computed.

For the computation of $\partial_{x_j}G_{2,i}(z,M(z,1,\pt{1}),\pt{1})$, we consider the derivatives of the equations in the original system. Notice, that at $\pt{x}=\pt{1}$ they collapse to equations in functions which we already computed (like $m_i(z)$ and $G_{2,i}(z,M(z,1,\pt{1}),\pt{1})$). For example the derivative of the first equation at $\pt{x}=\pt{1}$ equals
\begin{align*}
    g_{x_j} &= 2z(v+1)v_{x_j}g^2+2z(v+1)^2gg_{x_j}+zv_{x_j}(g+w) + zv(g_{x_j}+w_{x_j}) \\
    &+\sum_{i=0}^Ic(i)z^{e(i)}\prod_{j\in J_i} zQ_{b(i)-1}(z,v+1,g,m_0,\dots, m_\ell)T_j^{j_i}(z,m_0,\dots, m_\ell)
\end{align*}
The values of the partial derivatives of the involved functions at $z=1/12$ and $\pt{x}=\pt{1}$ can thus again be computed straight from the system.

\subsection{Proof of Lemma \ref{th:exp_var}} 
Analogously to our example in Section \ref{sec:koalas} we compute the expectation and variance of pattern counting variables by counting maps with distinguished faces. We start with the expectation, where we count maps with one distinguished simple $\ell_0$-gon and $n-d_0$ edges. 

As we recalled earlier, the expectation of face counting variables is given by the formula~(\ref{eq:exp_var_uni}). Now, the number of maps on $n-d_0$ edges with one marked simple $\ell_0$-gon equals
\[
    m_{n-d_0,e_{t(0)} = m_{n-d_0}}\left(f_{x_{t(0)}}(\pt{1})(n-d_0)+g_{x_{t(0)}}(\pt{1}) + O\left(\frac{1}{n}\right)\right)
\]
and the expectation of $X_n$ is therefore given by the expansion in (\ref{eq:exp_coeffs})
\begin{align*}
    \E[X_n] &= r_0\frac{m_{n-d_0, e_{t(0)}}}{m_{n-d_0}}\cdot \frac{m_{n-d_0}}{m_n}\\
    &= r_0\left(f_{x_{t(0)}}(\pt{1})(n-d_0)+g_{x_{t(0)}}(\pt{1}) + O\left(\frac{1}{n}\right)\right)\frac{[z^{n-d_0}]M(z,1,\pt{1})}{[z^n]M(z,1,\pt{1})}\\
    &= \frac{r_0}{12^{d_0}}\left(f_{x_{t(0)}}(\pt{1})(n-d_0)+g_{x_{t(0)}}(\pt{1}) + O\left(\frac{1}{n}\right)\right)\left(1+\frac{5d_0}{2n}+ O\left(\frac{1}{n^2}\right)\right)\\
    &= \frac{r_0}{12^{d_0}} f_{x_{t(0)}}(\pt{1})\left( n+\frac{3d_0}{2}+\frac{g_{x_{t(0)}}(\pt{1})}{f_{x_{t(0)}}(\pt{1})} + O\left(\frac{1}{n}\right)\right).
\end{align*}
Note that the factor $r_0$ appears because each rotation of the pattern in a map reduces to the same map with a distinguished simple $\ell_0$-gon.

We continue with the second factorial moment where we have to take into account all intersection types. First we count maps with either two labelled simple $\ell_0$-gons. There are $r_0^2$ ways to reinsert the $2d_0$ interior edges of the koalas. Thus, we count $r_0^2 m_{n-2d_0,2,0,0,0,0,0,0,0}$ maps in this case.
Then, we also count maps where we select a single face associated to an intersection type $i$. They are weighted by $r_i/c_{t(i)}$ by Lemma~\ref{prop:fact}.  
At the end, we label the two pattern occurrences again and gain a factor of $2!$. Naturally this translates into the equation
\begin{align*}
    \E[(X_n)_2] &= 2!\left(r_0^2\,\frac{m_{n-2d_0,2e_{t(0)}}}{m_{n-2d_0}}\,\frac{m_{n-2d_0}}{m_n} +\sum_{i=1}^I  \frac{r_i}{c_{t(i)}}\frac{m_{n-d_i,e_{t(i)}}}{m_{n-d_i}}\frac{m_{n-d_i}}{m_n} \right)\\
    &= 2!\left(r_0^2\,\E\left[Y^{(t(0))}_{n-2d_0}\left(Y^{(t(0))}_{n-2d_0}-1\right)\right] \,\frac{m_{n-2d_0}}{m_n} +\sum_{i=1}^I  \frac{r_i}{c_{t(i)}}\, \E\left[Y_{n-d_i}^{(t(i))}\right]\frac{m_{n-d_i}}{m_n} \right),
\end{align*}
where $\left(Y_n^{(1)},\dots, Y_n^{(J)}\right)$ is the random vector counting the number of faces which are marked by the variables $x_1, x_2,\dots, x_J$ respectively in a uniformly random planar map with $n$ edges.
Again, we use (\ref{eq:exp_var_uni}) and (\ref{eq:exp_coeffs}) such that we obtain
\begin{align*}
    \E[(X_n)_2] &= r_0^2f_{x_{t(0)}}(\pt{1})^2 (n-2d_0)^2\frac{m_{n-2d_0}}{m_n} \\
    &\quad+r_0^2\Big(\left(f_{x_{t(0)}x_{t(0)}}(\pt{1})  +  2f_{x_{t(0)}}(\pt{1})g_{x_{t(0)}}(\pt{1})\right)n+ O(1)\Big)\,\frac{m_{n-2d_0}}{m_n}\\
    &\quad+2\,\sum_{i=1}^I \left(\frac{r_i}{c_{t(i)}}f_{x_{t(i)}}(\pt{1})n + O\left(1\right)\right)\,\frac{m_{n-d_i}}{m_n}\\
    &= \frac{r_0^2f_{x_{t(0)}}(\pt{1})^2}{12^{2d_0}} (n-2d_0)^2\left(1+\frac{5d_0}{n} + O\left(\frac{1}{n^2}\right)\right) \\
    &\quad +\frac{r_0^2}{12^{2d_0}} \left(f_{x_{t(0)}x_{t(0)}}(\pt{1})+2f_{x_{t(0)}}(\pt{1})g_{x_{t(0)}}(\pt{1})\right)\Big(n+ O\left(1\right)\Big)\\
    &\quad +\sum_{i=1}^I \frac{2r_i}{12^{d_i}c_{t(i)}}f_{x_{t(i)}}(\pt{1})\;n+ O\left(1\right)
    \\
    &= \frac{r_0^2f_{x_{t(0)}}(\pt{1})^2}{12^{2d_0}} \,n^2 +\frac{r_0^2 \left(f_{x_{t(0)}x_{t(0)}}(\pt{1})+2f_{x_{t(0)}}(\pt{1})g_{x_{t(0)}}(\pt{1})+d_0f_{x_{t(0)}}(\pt{1})^2\right)}{12^{2d_0}}\,n\\
    &\quad +\sum_{i=1}^I \frac{2r_if_{x_{t(i)}}(\pt{1})}{12^{d_i}c_{t(i)}}\;n+O\left(1\right).
\end{align*}
Now, it is straight forward to compute the variance by
\begin{align*}
    \mathbb{V}[X_n] &= \E[X_n(X_n-1)]+\E[X_n] -\E[X_n]^2\\
    &= \frac{r_0^2f_{x_{t(0)}}(\pt{1})^2}{12^{2d_0}} \,n^2 +\frac{r_0^2 \left(f_{x_{t(0)}x_{t(0)}}(\pt{1})+2f_{x_{t(0)}}(\pt{1})g_{x_{t(0)}}(\pt{1})+d_0f_{x_{t(0)}}(\pt{1})^2\right)}{12^{2d_0}}\,n\\
    &\quad +\sum_{i=1}^I \frac{2r_if_{x_{t(i)}}(\pt{1})}{12^{d_i}c_{t(i)}}\;n + \frac{r_0f_{x_{t(0)}}(\pt{1})}{12^{d_0}}  \;n \\
    &\quad- \frac{r_0^2f_{x_{t(0)}}(\pt{1})^2}{12^{2d_0}} \left( n+\frac{3d_0}{2}+\frac{g_{x_{t(0)}}(\pt{1})}{f_{x_{t(0)}}(\pt{1})} + O\left(\frac{1}{n}\right)\right)^2\\
    &= \frac{1}{12^{2d_0}}\left(r_0^2f_{x_{t(0)}x_{t(0)}}(\pt{1})-2d_0r_0^2f_{x_{t(0)}}(\pt{1})^2 +12^{d_0}r_0f_{x_{t(0)}}(\pt{1})\right)n\\
    &\quad+\sum_{i=1}^I \frac{2r_if_{x_{t(i)}}(\pt{1})}{12^{d_i-2d_0}c_{t(i)}}\;n+O\left(1\right).
\end{align*}

\subsection{Proof of Theorem \ref{thm:central}}\label{sec:proofcent}
Finally, we prove the main Theorem~\ref{thm:central}.

Analogously to Section \ref{sec:koalas} we compute the $k$-th factorial moment, where $k = O(\sqrt{n})$ and prove  asymptotic estimates of the same form as in Theorem~\ref{thm:gao_wormald} such that we can derive our main result as a direct consequence.

Recall that the $k$-th factorial moment is the number of maps on $n$ edges with $k$ selected pattern occurrences which are labelled from $1$ to $k$ divided by the total number of maps on $n$ edges. By Lemmas~\ref{lem:S1} and \ref{th:exp_var}, we may further restrict ourselves to count maps on $n$ edges with $k$ selected and labelled pattern occurrences where each selected pattern intersects with at most one other selected pattern occurrence. Hence, 
\begin{equation*}
    \E\left[(X_n)_k\right] = \frac{m_{n,k}^\circ}{m_n} \sim  \frac{m_{n,k}^{\circ,\times}}{m_n} 
\end{equation*}
where $m_{n,k}^\circ$ is the number of maps with $n$ edges and $k$ labelled pattern occurrence (among arbitrary many occurrences) and $m_{n,k}^{\circ,\times}$ is the subset of all such maps where each labelled pattern occurrence intersects at most one other labelled pattern occurrence.

For a map in $m_{n,k}^{\circ,\times}$, let $p_i$ be the the number of pairs of intersecting labelled pattern occurrences with (rooted) intersection type $i$ and let $p_0$ be the number of labelled pattern occurrences which do not intersect with any other. Of course it has to hold $p_0+\sum 2p_i = k$. 

Now instead of counting maps with $n$ edges with $p_0$ single labelled patterns and $p_i$ intersecting pairs of labelled patterns with intersection type $i$, $1\leq i\leq I$, we unlabel the patterns and delete the $d_i$ edges according to the rules we set for intersection type $i$. Thus, we end up with a map on 
\[
    n-D := n-p_0d_0 - \sum p_id_i
\]
edges and $p_0+\sum p_i$ distinguished faces. We already have counting variables for these faces in our generating function defined in~(\ref{eq:gen_pats}). However, several intersection types might be associated to the same face. So for each $j\in [J]$, we have 
\[
    P_j := \sum_{i\in t^{-1}(j)} p_{i}
\]
distinguished faces counted by the variable $x_{j}$.
In turn, we know that the number of maps with $n-D$ edges and $(P_1,P_2,\dots,P_J)$ distinguished faces is given by
\begin{align*}
    m_{n-D, P_1,P_2,\dots,P_J} = \frac{m_{n-D}}{P_1!P_2!\cdots P_T!}\E \left[ Y_{n-D}^{(P_1)}Y_{n-D}^{(P_2)}\cdots Y_{n-D}^{(P_J)}\right]
\end{align*}
and we can determine $\E \left[ Y_{n-D}^{(P_1)}Y_{n-D}^{(P_2)}\cdots Y_{n-D}^{(P_J)}\right]$ by Lemma~\ref{LeMm}.

Before we go into the details of this further computation, we establish how many maps with $k$ selected and labelled patterns with $(p_0,p_1,p_2,\dots,p_I)$ intersection types will reduce to this map with $n-D$ edges and $(P_1,P_2,\dots, P_J)$ selected faces. First of all, if for $j\in J$, the preimage of $j$,
\[
    t^{-1}(j) = \{i_1,i_2,\dots,i_{j}\}
\]
is not a single integer, then we partition the  $P_j$ selected faces into $p_{i_1},\dots,p_{i_j}$ faces and attribute them to the the intersection types $i_1, i_2, \dots, i_j$ respectively. Again, we gain a factor of $\left(\frac{r_i}{c_{t(i)}}\right)^{p_i}$ for each intersection type according to Lemma~\ref{prop:fact}. In case $0 \in t^{-1}(j)$, then a factor $r_0$ for each rotation of a single pattern appears. That is, in total we get the factor
\[
    r_0^{p_0}\;\left(\frac{r_i}{c_{t(i)}}\right)^{p_i} \binom{P_j}{p_{i_1},p_{i_2},\dots,p_{i_j}} = \frac{r_0^{p_0}}{p_0!} \; \left(\prod_{i= 1}^{I} \frac{r_i^{p_i}}{c_{t(i)}^{p_i}p_i!}\right)\left(\prod_{j=1}^{J} P_j!\right).
\]

Finally, let us not forget that we have to relabel the $k$ selected patterns such that we get in total

\begin{align*}
    \frac{m_{n,k}^{\circ,\times}}{m_n} &= k!\;\sum_{p_0+2\sum p_i = k} \frac{r_0^{p_0}}{p_0!} \; \prod_{i=1}^{I} \frac{r_i^{p_i}}{c_{t(i)}^{p_i}p_i!} \prod_{j=1}^{J} P_j!\;\frac{m_{n-D, P_1, P_2,\dots,P_j}}{m_n}\\
    &= k!\;\sum_{p_0+2\sum p_i = k} \frac{r_0^{p_0}}{p_0!} \prod_{i=1}^{I} \frac{r_i^{p_i}}{c_{t(i)}^{p_i}p_i!} \;\frac{m_{n-D}}{m_n} \;\E \left[ Y_{n-D}^{(P_1)}Y_{n-D}^{(P_2)}\cdots Y_{n-D}^{(P_J)}\right].
\end{align*}
Now we restructure the sum a bit by writing $p_0 = k-2\,\sum_{i=1}^I p_i$  
\begin{align*}
    \frac{m_{n,k}^{\circ,\times}}{m_n} &= k!\sum_{P = 0}^{\lfloor k/2 \rfloor} \sum_{\sum p_i = P}
    \frac{r^{k-2P}}{(k-2P)!} \prod_{i=1}^{I} \frac{r_i^{p_i}}{c_{t(i)}^{p_i}p_i!}
    \;\frac{m_{n-D}}{m_n} \;\E \left[ Y_{n-D}^{(P_1)}Y_{n-D}^{(P_2)}\cdots Y_{n-D}^{(P_J)}\right]
\end{align*}
and then we apply Lemma~\ref{LeMm} to obtain
\begin{align*}
    \frac{m_{n,k}^{\circ,\times}}{m_n} &= k!\sum_{P = 0}^{\lfloor k/2 \rfloor} \sum_{\sum p_i = P}
    \frac{m_{n-D}}{m_n} \frac{\left(r_0f_{x_{t(0)}}(\pt{1})\right)^{k-2P}}{(k-2P)!} \prod_{i=1}^{I} \frac{1}{p_i!}\left(\frac{r_if_{x_{t(i)}}(\pt{1})}{c_{t(i)}}\right)^{p_i}\\
    &\hspace{30mm}\cdot 
     (n-D)^{k-P}\exp\left(\left(\frac{1}{2(n-D)} \left<\bar{\pt{p}},\Sigma  \bar{\pt{p}} \right>\right)\right) \left(1+O\left(\frac{D}{n}\right)\right)
\end{align*}
where  $\bar{\pt{p}} = (k-2P,p_1,p_2,\dots p_T)^T$ and $(\Sigma)_{i+1,j+1} = \frac{f_{x_{t(i)}x_{t(j)}}(\pt{1})}{f_{x_{t(i)}}(\pt{1})f_{x_{t(j)}}(\pt{1})}$. 
Note that Lemma~\ref{LeMm} actually gives an expression with $\left<\pt{v}, \bar \Sigma \pt{v}\right>$ in the exponential function, where $v=(v_1,v_2,\dots v_\ell)$ and $v_i$ is the sum of all $p_j$ with $t(j) = i$ while $(\bar \Sigma)_{i,j} = \frac{f_{x_ix_j}(\pt{1})}{f_{x_i}(\pt{1})f_{x_j}(\pt{1})}$. However it is easy to see by elementary computations that we can expand this expression to $\left<\bar{\pt{p}},\Sigma  \bar{\pt{p}} \right>$. 

Next, we apply Lemma~\ref{le:frac_map} and obtain
\begin{align*}
    \frac{m_{n,k}^{\circ,\times}}{m_n} 
    &= \frac{k!}{12^D}\sum_{P = 0}^{\lfloor k/2 \rfloor} \sum_{\sum p_i = P}
    \frac{\left(r_0f_{x_{t(0)}}(\pt{1})\right)^{k-2P}}{(k-2P)!} \prod_{i=1}^{I} \frac{1}{p_i!}\left(\frac{r_if_{x_{t(i)}}(\pt{1})}{c_{t(i)}}\right)^{p_i}\\
    &\hspace{30mm}\cdot 
     (n-D)^{k-P}\exp\left(\left(\frac{1}{2(n-D)} \left<\bar{\pt{p}},\Sigma  \bar{\pt{p}} \right>\right)\right) \left(1+O\left(\frac{D}{n}\right)\right).
\end{align*}

Now that we have all asymptotic estimates of the individal terms, we aim to identify which ones drive the asymptotic growth of the sum. As we expect asymptotics of the form in Theorem~\ref{thm:gao_wormald}, we factor out the main asymptotic term of $\mu_n^k$ which has already appeared in the sum. Further, we distribute the factors of $(n-D)^{-P}$ among the individal $p_i$'s. That is,
\begin{align*}
    \frac{m_{n,k}^{\circ,\times}}{m_n} &= \left(\frac{r_0}{12^{d_0}}f_{x_{t(0)}}(\pt{1})n\right)^{k} \\
    &\quad \cdot \sum_{P = 0}^{\lfloor k/2 \rfloor} \sum_{\sum p_i = P}
    \frac{k!}{12^{D-kd_0}(k-2P)!}\left(1-\frac{D}{n}\right)^k\prod_{i=1}^I\frac{1}{p_i!}\left(\frac{r_if_{x_{t(i)}}(\pt{1})}{r_0^2c_{t(i)}(n-D)f_{x_{t(0)}}(\pt{1})^2}\right)^{p_i} \\
    &\hspace{27mm}\cdot 
    \exp \left( \left(\frac{1}{2(n-D)} \left<\bar{\pt{p}},\Sigma  \bar{\pt{p}} \right>\right)\right)\left(1+O\left(\frac{D}{n}\right)\right).
\end{align*}
The next step is analogous to our computations in Section~\ref{sec:koalas} where we rewrite 
\begin{align*}
    \left<\bar{\pt{p}},\Sigma \bar{\pt{p}}\right> 
    = k^2\frac{f_{x_{t(0)}x_{t(0)}}(\pt{1})}{f_{x_{t(0)}}(\pt{1})f_{x_{t(0)}}(\pt{1})}  + \sum_{i=1}^I q_i p_i.
\end{align*}
Recall that $q_i$ is linear in $P$ and $k$ (without mixed terms) according to this calculation. Hence,
\[
    \exp\left(\frac{1}{2(n-D)} \sum_{i=1}^I q_i p_i\right) = \prod_{i=1}^I\left(1+O\left(\frac{k}{n}\right)\right)^{p_i}. 
\]
Further we factor out more terms independent of any $p_i$'s and since $D = O(k)$, the above simplifies to
\begin{align*}
    \frac{m_{n,k}^{\circ,\times}}{m_n} &= \left(\frac{r_0}{12^{d_0}}f_{x_{t(0)}}(\pt{1})n\right)^{k}\exp\left(\frac{k^2}{2n} \frac{f_{x_{t(0)}x_{t(0)}}(\pt{1})}{f_{x_{t(0)}}(\pt{1})f_{x_{t(0)}}(\pt{1})}\right) \\
    &\qquad \cdot \sum_{P = 0}^{\lfloor k/2 \rfloor} \sum_{\sum p_i = P}\Bigg(
    \frac{k!}{12^{D-kd_0}(k-2P)!}\left(1-\frac{D}{n}\right)^k\\
    &\hspace{35mm}\cdot \prod_{i=1}^I\frac{1}{p_i!}\left(\frac{r_if_{x_{t(i)}}(\pt{1})}{r_0^2c_{t(i)}nf_{x_{t(0)}}(\pt{1})^2}\left(1+O\left(\frac{k}{n}\right)\right)\right)^{p_i}\Bigg).
\end{align*}

By now, we can see a pattern emerge reminiscent of the product of several exponential sums. 
We distinguish between the terms where $P\le k^{1/3}$, and those where $P > k^{1/3}$.
In the first case ($P\le k^{1/3}$) we have
$
     \frac{k!}{(k-2P)!}= k^{2P}\left(1+O\left(\frac{1}{k^{1/3}}\right)\right)
$
and (with $D = p_0 d_0 + \sum d_i p_i = d_0(k-2P) + \sum d_i p_i$) 
\[
  \left( 1 - \frac Dn  \right)^k  =  \left(1-\frac{d_0(k-2P)+\sum_{i=1}^m p_id_i}{n}\right)^k = \exp\left(-\frac{d_0k^2}{n}\right)\left(1+O\left(\frac{1}{k^{2/3}}\right)\right)
\] 
whereas in the remaining cases we use the trivial upper bounds $\frac{k!}{(k-2P)!}\le k^{2P}$
and 
\[
   \left( 1 - \frac Dn  \right)^k  =   \left(1-\frac{d_0(k-2P)+\sum_{i=1}^m p_id_i}{n}\right)^k \le 1.
\] 
In particular for $P > k^{1/3}$ we observe that
\begin{align*}
  & \sum_{P > k^{1/3}} \sum_{\sum p_i = P} \Bigg(
    \frac{k!}{12^{D-kd_0}(k-2P)!}\left(1-\frac{D}{n}\right)^k \prod_{i=1}^I\frac{1}{p_i!}\left(\frac{r_if_{x_{t(i)}}(\pt{1})}{r_0^2c_{t(i)}nf_{x_{t(0)}}(\pt{1})^2}\left(1+O\left(\frac{k}{n}\right)\right)\right)^{p_i}\Bigg) \\
&\le \sum_{P > k^{1/3}} \sum_{\sum p_i = P} \frac 1{p_i!}
  \left(\frac{k^2}{n}\frac{12^{d_i-2d_0}r_if_{x_{t(i)}}(\pt{1})}{r_0^2c_{t(i)}nf_{x_{t(0)}}(\pt{1})^2}\left(1+O\left(\frac{1}{k}\right)\right)\right)^{p_i} \\
&\le \sum_{P > k^{1/3}} \frac 1{P!} 
 \left( \frac{k^2}n \sum_{i=1}^I \frac{12^{d_i-2d_0}r_if_{x_{t(i)}}(\pt{1})}{r_0^2c_{t(i)}nf_{x_{t(0)}}(\pt{1})^2}\left(1+O\left(\frac{1}{k}\right)\right)  \right)^P \\
 &\le \sum_{P > k^{1/3}} \frac {A^P}{P!} = o(1)
\end{align*}
for a proper constant $A > 0$.
Hence, we finally obtain 
\begin{align*}
    \frac{m_{n,k}^{\circ,\times}}{m_n} &\sim \left(\frac{r_0}{12^{d_0}}f_{x_{t(0)}}(\pt{1})n\right)^{k}\exp\left(\frac{k^2}{2n}\left( \frac{f_{x_{t(0)}x_{t(0)}}(\pt{1})}{f_{x_{t(0)}}(\pt{1})f_{x_{t(0)}}(\pt{1})} - 2d_0\right)\right)\left(1+O\left(\frac{1}{k^{2/3}}\right)\right) \\
    &\qquad \cdot \sum_{P \le k^{1/3}}  \sum_{\sum p_i = P}
    \prod_{i=1}^I\frac{1}{p_i!}\left(\frac{k^2}{n}\frac{12^{d_i-2d_0}r_if_{x_{t(i)}}(\pt{1})}{r_0^2c_{t(i)}nf_{x_{t(0)}}(\pt{1})^2}\left(1+O\left(\frac{1}{k^{1/3}}\right)\right)\right)^{p_i}\\ 
    &\sim \left(\frac{r_0}{12^{d_0}}f_{x_{t(0)}}(\pt{1})n\right)^{k}\exp\left(\frac{k^2}{2n}\left( \frac{f_{x_{t(0)}x_{t(0)}}(\pt{1})}{f_{x_{t(0)}}(\pt{1})^2} - 2d_0+
    2\sum_{i=1}^I\frac{12^{d_i-2d_0}r_if_{x_{t(i)}}(\pt{1})}{r_0^2c_{t(i)}f_{x_{t(0)}}(\pt{1})^2}\right)\right)\\
    &\sim (\mu_n)^k\exp\left(\frac{k^2}{2n}\left( \frac{f_{x_{t(0)}x_{t(0)}}(\pt{1})}{f_{x_{t(0)}}(\pt{1})^2} - 2d_0+
    2\sum_{i=1}^I\frac{12^{d_i-2d_0}r_if_{x_{t(i)}}(\pt{1})}{r_0^2c_{t(i)}f_{x_{t(0)}}(\pt{1})^2}\right)\right)
\end{align*}
which is the desired form.

Now by Lemma \ref{lem:S1} we can conclude that
\[
    \E[(X_n)_k] \sim \frac{m_{n,k}^{\circ,\times}}{m_{n}}
\]
and indeed, by Lemma \ref{th:exp_var}, the expression in the exponential function equals exactly the desired 

\[
    \frac{k^2}2 \frac{ \sigma_n^2 - \mu_n}{\mu_n^2} \sim \frac{k^2}{2n}\left( \frac{f_{x_{t(0)}x_{t(0)}}(\pt{1})}{f_{x_{t(0)}}(\pt{1})^2} - 2d_0+
    2\sum_{i=1}^I\frac{12^{d_i-2d_0}r_if_{x_{t(i)}}(\pt{1})}{r_0^2c_{t(i)}f_{x_{t(0)}}(\pt{1})^2}\right)
\]
such that 
\[
    \E[(X_n)_k] \sim \mu_n^k \exp\left( \frac{k^2}2 \frac{ \sigma_n^2 - \mu_n}{\mu_n^2}  \right) 
\]
and 
Theorem \ref{thm:gao_wormald} finally yields Theorem \ref{thm:central}.

\section{Extensions to other map families and open questions}

There are several sub-classes of planar maps that have been discussed in the literature, 
for example $2$- or $3$-connected planar maps, bipartite or Eulerian planar maps, loopless or simple planar maps, 
triangulations or $3$-regular planar maps, quadrangulations or $4$-regular planar maps etc. 
It is therefore a natural question whether
Theorem~\ref{thm:central} can be extended to these classes of planar maps, in particular
by applying a corresponding approach.

The main ingrediences of the proof of Theorem~\ref{thm:central} are the following 
two properties. The first one is a combinatorial one, the second is an analytic one.
Let $\mathcal{M}$ denote a class of planar maps that we want to consider:
\begin{enumerate}
\item Suppose that ${\bf p}\in \mathcal{M}$ is a pattern with $\ell$ interior edges and outer face valency $r$,
and ${\bf m}^{({\bf p})}$ is a map in $\mathcal{M}$ with $n$ edges, where
the occurrence of the pattern ${\bf p}$ is selected. If we delete the inner edges of the selected pattern
then we obtain a planar map in $\mathcal{M}$ with $n-\ell$ edges, where a face of valency $r$ is selected.
Furthermore this property generalizes to several selected patterns.
\item Let $Y_n$ denote the (random) number of simple $r$-gons in a random planar map of size $n$ in $\mathcal{M}$,
where we assume that every planar map of size $n$ is equally likely. Then the probability generating function
$\mathbb{E}[x^{Y_n}]$ satisfies the property (\ref{eq:power}) for complex $x$ sufficiently close to $1$.
Furthermore this property generalizes to the random vector that counts faces of several shapes.
\end{enumerate}
If these two properties are satisfied then (by the first property) 
we can relate the factorial moments of the random variable $X_n$ 
that counts the number of occurrences of a (given) pattern ${\bf p}$ to factorial moments of $Y_n$ and 
(by the second property) we get asymptotics for the factorial moments of $Y_n$ (by Lemma~\ref{LeMm}).
Putting everything together we get corresponding asymptotics for the factorial moments of $X_n$
and derive asymptotic normality of $X_n$.

The first property is certainly satisfied for $2$-connected, bipartite, loopless, and simple planar maps.
However, it is not satisfied for Eulerian, $3$-connected, triangulations, quadrangulations or for
$3$- or $4$-regular planar maps.

In order to check the second property we can use the method from \cite{DrPa} (see also \cite{Yu}).
In this work, a proper generalization of the quadratic method is applied in order to obtain asymptotics
for $\mathbb{E}[x^{Y_n}]$ of the form (\ref{eq:power}) which implies then directly a central limit
theorem for $Y_n$ in the case of all planar maps as well as for $2$-connected planar maps.
In \cite{Yu} this method was generalized to simple planar maps. It is an easy exercise to generalize
this method to bipartite and to loopless planar maps. Furthermore it also extends to several 
face valencies. 

Summing up we can expect a central limit theorem for pattern counts in $2$-connected, bipartite, loopless, and simple planar maps
by applying the same approach as above.

Another extension of the result concerns submaps that leave a single face (not necessarily with simple boundary) if one deletes several interior edges. An example would be the submap of two intersecting koalas in intersection type 2 in Figure \ref{fig:intersects}. Naturally, the method directly extends to these situations. However, they are not patterns according to Definition 3.

Finally we want to mention one further potential extension.
Suppose that ${\bf p}_1, \ldots {\bf p}_d$ are $d$ different patterns. Then it is natural
to consider the random vector ${\bf Y}_n = (Y_n^{(1)}, \ldots, Y_n^{(d)})$, where $Y_n^{(j)}$ counts the
number of occurrences of the patterns ${\bf p}_j$ in planar maps of size $n$. 
It is clearly expected that ${\bf Y}_n$ satisfies a multivariate central limit theorem 
(similar to face valency counts). And actually we could follow the same lines as above since 
both properties are again satisfied. However, in order to derive a central limit theorem for ${\bf Y}_n$
we need a multivariate extension of Theorem~\ref{thm:gao_wormald} by Gao and Wormald \cite{GaoWormald}
that is actually provided independently by Ojeda, Holmgren, and Janson\cite{ojeda2023fringe} and by Hitczenko and Wormald \cite{hitczenko2023multivariate}. This extension is quite feasible 
using the approach of \cite{GaoWormald} but a bit complicated. It should be straight forward but there are covariance conditions that have to be verified. 


\bibliography{Bib_maps.bib}{}
\bibliographystyle{plain}
\end{document}